\documentclass[12pt,a4paper]{article}
\usepackage{graphicx} 
\usepackage{amsmath}
\usepackage{amsthm}
\usepackage{authblk}
\usepackage{bbm}
\usepackage[margin=1in]{geometry} 
\usepackage{amssymb}
\usepackage{amsfonts}
\usepackage{float}
\usepackage[hidelinks]{hyperref}
\usepackage{url}
\usepackage{mathrsfs}
\usepackage{xcolor}
\newtheorem{theorem}{Theorem}
\newtheorem{lemma}[theorem]{Lemma}

\newtheorem{remark}{Remark}

\newcommand{\isep}{\mathrel{{.}\,{.}}\nobreak}

\renewcommand{\l}{\ell}

\newcommand{\E}{\mathbb{E}}

\newcommand{\ep}{\varepsilon}
\renewcommand{\P}{\mathbb{P}}

\title{On the Group Randomness of 0-1 Real Sequences from Binary Linear Codes}
\author{Chin Hei Chan}
\affil{\small{Hetao Institute of Mathematics and Interdisciplinary Sciences, Shenzhen 518017, Guangdong, China}}

\affil[ ]{\href{mailto:chenzhanxi@himis-sz.cn}{\texttt{chenzhanxi@himis-sz.cn}}}

\date{}

\begin{document}
\maketitle
\begin{abstract}
    In this paper, we study the group randomness of 0-1 real sequences derived from a binary linear code through the spectral anlysis of a suitable normalization of the Gram matrix of a $p \times n$ random matrix whose rows are uniformly and independently drawn from those 0-1 real sequences, where $p/n \to y \in (0,1)$. We show that as $n \to \infty$, its empirical spectral distribution converges to the Marchenko-Pastur law at a rate at least of the order $n^{-1/4}$ with high probability as a consequence of eigenvalue interlacing with the $\pm 1$ sequence model studied in 2021 \cite{CTX}, and that its largest eigenvalue, after centering by $p+1+y$, converges to a Gaussian random variable with mean 0 and variance $4y$, provided that the dual distance of the code is at least 5. Furthermore, we show that when this is the only control parameter, the lower bound 5 is optimal in general by demonstrating that the binary Simplex codes and first-order Reed-Muller codes of dual distance 3 and 4 respectively do not possess either property.
\end{abstract}


\section{Introduction}
Group randomness is the term to describe joint randomness of a given class of real or complex sequences. This means that random matrices constructed from these sequences resemble properties satisfied by truly random matrices, that is, matrices whose entries are independently and identically distributed (i.i.d.). These matrices play a central role in signal processing applications such as compressed sensing \cite{CHJ}.

Group randomness was introduced by Babadi and Tarokh in \cite{Babadi,Babadi2}, who considered a Gram matrix constructed by picking at random $\pm 1$ real sequences derived from binary linear codes. They observed through simulation experiments that the empirical spectral distributions (ESDs) of the Gram matrices based on BCH codes and Gold sequences converge to the Marchenko-Pastur (MP) law as the length increases to infinity, a phenomenon satisfied by truly random sample covariance matrices, while that from pseudo-noise (PN) sequences does not. Xia and Xiong \cite{OQBT} confirmed this by showing that such behaviour holds in expectation provided that the minimum Hamming distance of the dual code is at least 5, which is achieved by the odd-degree Gold codes exactly. Later, the authors in \cite{CTX} proved that the convergence rate is at least of the order $n^{-1/4}$ with high probability, where $n$ is the code length. In a more recent paper \cite{CLT}, the authors proved an advanced group randomness property, that the central limit theorem holds for linear spectral statistics of monomials on the same random matrix over binary linear codes with dual distance at least 7. 

In this paper we study the group randomness of 0-1 real sequences derived from binary linear codes. While this construction looks simpler and more intuitive, the entries of the matrix whose rows are randomly chosen from these 0-1 sequences are non-central, unlike the previous construction. Nevertheless, we show that provided that the dual distance of the code is at least 5, the Gram matrix of our code matrix resembles spectral behaviour of non-central i.i.d. sample covariance matrices. One of them is that its ESD converges to the MP law with high probability at a rate at least of the order $n^{-1/4}$, which follows from the $\pm 1$-based code matrix model, as a consequence of eigenvalue interlacing. Yet we also establish another group randomness phenomenon: its largest eigenvalue grows proportionally with the matrix dimension and exhibits asymptotically Gaussian fluctuation as $n \to \infty$, which was previously established for general non-central i.i.d. input matrices by Silverstein \cite{Silverstein} and general independent entry matrices by Bryc and Silverstein \cite{Bryc}.

Before stating our results and comparing the two different code-based random matrices in more detail, we introduce some notation.

Let $\mathcal{C}$ be a binary linear code of length $n$. Its dual code $\mathcal{C}^\bot$ consists of codewords which are orthogonal to all codewords of $\mathcal{C}$ under the usual inner product defined in $\mathbb{F}_2^n$. $\mathcal{C}^\bot$ is also a binary linear code of length $n$. The dual distance $d^\bot$ of the original code $\mathcal{C}$ is the minimum (Hamming) distance of $\mathcal{C}^\bot$.

The $\pm 1$ matrix model in \cite{CTX,OQBT} is constructed as follows. Let $\psi: \mathbb{F}_2 \to \mathbb{R}, a \mapsto (-1)^a$ be the standard additive character, and extend component-wise to $\mathbb{F}_2^n$. For a positive integer $p < n$, let $\Phi_\mathcal{C}$ be a $p \times n$ random matrix whose rows are uniformly and independently drawn from $\psi(\mathcal{C})$. Let $\mathcal{G}_\mathcal{C}$ be the Gram matrix of $\frac{1}{\sqrt{n}}\Phi_\mathcal{C}$, that is,
$$\mathcal{G}_\mathcal{C}=\frac{1}{n}\Phi_\mathcal{C}\Phi_\mathcal{C}^T.$$
Next, we state the notations and main result with regard to the 0-1 code random matrix model studied in this paper.

 For $p < n$, let $\widetilde{\Phi}_\mathcal{C}$ be a $p \times n$ random matrix, whose rows are uniformly and independently picked from $\mathcal{C}$, treating 0 and 1 as real values. In order for the variance to align with the $\pm 1$ matrix setting, we consider the Gram matrix of $\frac{2}{\sqrt{n}}\widetilde{\Phi}_\mathcal{C}$, that is,
$$\widetilde{\mathcal{G}}_\mathcal{C}=\frac{4}{n}\widetilde{\Phi}_\mathcal{C}\widetilde{\Phi}_\mathcal{C}^T.$$
Let $\tilde{\lambda}_1$ be the largest eigenvalue of $\widetilde{\mathcal{G}}_\mathcal{C}$. 

Our main result is as follows.
\begin{theorem}\label{mainthm2}
    Let $\{\mathcal{C}_n\}$ be a sequence of binary linear codes of length $n$ and dual distance $d_n^\bot$. Let $\widetilde{\Phi}_{\mathcal{C}_n}$ be the $p_n \times n$ 0-1 random matrix as constructed above, with $y_n=p_n/n \to y \in (0,1)$ as $n \to \infty$. Let $\widetilde{\mathcal{G}}_{\mathcal{C}_n}$ be the Gram matrix of $\frac{2}{\sqrt{n}}\widetilde{\Phi}_{\mathcal{C}_n}$, and $\tilde{\lambda}_1$ denote the largest eigenvalue of $\widetilde{\mathcal{G}}_{\mathcal{C}_n}$. If $d_n^\bot \geq 5$ for all $n$, then as $n \to \infty, \tilde{\lambda}_1-(p_n+1+y_n)$ converges in distribution to the Gaussian distribution with mean $0$ and variance $4y$.
\end{theorem}
\subsection{Discussion of the Main Theorem}
The behaviour of the largest eigenvalue of the two matrix models differs significantly. Numerical simulations from BCH codes and Gold codes suggest that the largest eigenvalue of $\mathcal{G}_\mathcal{C}$ appears to be of order 1, and usually lies close to the upper bound of the support of the MP law (see the figures shown in \cite{CTX} for instance). Note that this has only been proved mathematically to be the case for truly random matrices (in particular this covers $\mathcal{G}_\mathcal{C}$ with $\mathcal{C}$ being the full binary code $\mathbb{F}_2^n$). Moreover, its fluctuation is asymptotically Tracy-Widom \cite{Bloemendal,Pillai} in the latter case. However, Theorem \ref{mainthm2} states that the largest eigenvalue of the (non-central) $\widetilde{\mathcal{G}}_\mathcal{C}$ is much larger (in fact has the same order as the dimension of the matrix), and also yields an asymptotically Gaussian fluctuation. The latter is known to be the case for several spiked sample covariance matrix models with independent entries, provided that the largest eigenvalue of the spike exceeds a certain threshold \cite{BY,BBP,BS,Bryc,Johnstone,LV,Paul,Silverstein}. In particular, \cite{Bryc,Silverstein} implicitly cover our model for the full binary code case as well. However, we prove an analogous result under the substantially weaker assumption $d^\bot \geq 5$, thereby revealing a novel group randomness phenomenon in terms of a single (outlier) eigenvalue.

Another interesting thing here is that the lower bound of the sufficient dual distance condition in the ESD results of both code matrix models (see Theorems \ref{mainthm0} and \ref{mainthm} in Sections \ref{Pre} and \ref{Grelate} respectively) as well as Theorem \ref{mainthm2} are the same, namely 5. It essentially means 4-wise independence for each row of the input code matrix. This condition also appears in \cite{CKX,CX2} for the group randomness of the $\pm 1$ binary sequences in another normalization for the Wigner semicircle law. We remark that when the dual distance is the only parameter to control, this lower bound 5 is optimal since the binary first-order Reed-Muller (RM) codes with $d^\bot=4$ do not possess any of these group randomness properties. However, the condition in Theorem \ref{mainthm2} can actually be relaxed to $d^\bot \geq 3$ (pairwise independence within row entries) together with $B_4=o(n^3)$ where $B_4$ is the number of codewords of weight 4 in $\mathcal{C}^\bot$ (see Theorem \ref{mainthm2'} in Section \ref{pthm2}), where the latter is not satisfied by the binary first-order RM codes.

The main idea behind the proof of Theorem \ref{mainthm2} (or more generally, Theorem \ref{mainthm2'}) is as follows: we show that $\tilde{\lambda}_1-\|\mathbf{u}\|^2$ converges to 1 in probability as length $n$ increases to infinity (Lemma \ref{lemLUS} in Section \ref{pthm2}). Here $\mathbf{u}$ is the mean-normalized weight vector of the independently chosen codewords, which is related to the dominant rank-1 component $\mathbf{u}\mathbf{u}^T$ of $\widetilde{\mathcal{G}}_\mathcal{C}$ relative to $\mathcal{G}_\mathcal{C}$. This holds provided that $d^\bot \geq 3$ (or equivalently, pairwise independence within row entries), and reduces the problem to studying the limiting mean and distribution of $\|\mathbf{u}\|^2$, which is a sum of $p$ i.i.d. random variables governed by the weight distribution of $\mathcal{C}$. While this idea bears some similarity with \cite{Bryc}, we work on the normalized Gram matrix $\widetilde{\mathcal{G}}_\mathcal{C}$ directly instead of the input matrix $\widetilde{\Phi}_\mathcal{C}$. \cite{Bryc} assumes that all entries of the input matrix are independent, while our eigenvalue approximation in Lemma \ref{lemLUS} only needs pairwise independence within a row. However this alone is not sufficient to guarantee that the Gaussian central limit theorem on $\|\mathbf{u}\|^2$ holds. This is where the extra condition $B_4=o(n^3)$ enters: it provides the fourth-moment control needed to verify Lyapunov's condition on the centered weights of codewords in Lemma \ref{lemU} (Section \ref{pthm2}). This also relates to the term ``group randomness'', which now refers to the joint randomness of codewords in terms of weight distribution.

For the binary Simplex codes and first-order RM codes where $B_4$ is precisely of the order $n^3$, we verify that the limiting fluctuation of $\tilde{\lambda}_1$ is a linear combination of Poisson random variables governed by the rare weight codewords, namely $\mathbf{0}$ and $\mathbf{1}$ (Theorem \ref{thmLS} in Section \ref{SRM}). This shows that the condition on $B_4$ in Theorem \ref{mainthm2'} cannot be dropped in general, and that if $d^\bot$ is the only control parameter, the lower bound 5 in Theorem \ref{mainthm2} cannot be further reduced. In addition, we remark that the asymptotic equation about the largest singular value in \cite[Proposition 3.1]{Bryc} does not hold with the error term $\ep^{(n)} \to 0$ in probability for those two classes of codes, let alone the asymptotic non-Gaussianity of $Z_1$. This motivates us to directly work on $\widetilde{\mathcal{G}}_\mathcal{C}$ instead.

\subsection{Organization of the Paper}
The paper is organized as follows. We first introduce the notations and some elementary but useful results in Section \ref{Pre}. Then in Section \ref{Grelate} we show that the eigenvalues of $\widetilde{\mathcal{G}}_\mathcal{C}$ and $\mathcal{G}_\mathcal{C}$ interlace with each other, and thereby the ESD convergence rate result (Theorem \ref{mainthm}) from this. Sections \ref{pthm2}--\ref{plemLUS} are devoted to the proof of Theorem \ref{mainthm2}. To streamline the presentation, in Section \ref{pthm2} we prove Theorem \ref{mainthm2} (more precisely, Theorem \ref{mainthm2'}) using Lemmas \ref{lemU} and \ref{lemLUS}, whose proofs are deferred to Section \ref{plemLUS}. In Section \ref{SRM} we prove that the limiting ESD cannot be MP law and derive an explicit Poisson-type limiting distribution of $\|\mathbf{u}\|^2$ for the binary Simplex codes and first-order RM codes (Theorems \ref{thmRM} and \ref{thmLS} respectively). In Section \ref{ns}, we present our numerical simulations based on binary Gold codes, Simplex codes and first-order RM codes. Finally in Section \ref{op}, we conclude the paper and discuss several related open problems.

\section{Preliminaries}\label{Pre}
\subsection{Notations}
Throughout this paper, we adopt the following notations:
\begin{itemize}
    \item $\mathcal{C}$ is an $[n,k]$ binary linear code;
    \item $d^\bot$ is the dual distance of $\mathcal{C}$ (or equivalently, the minimum distance of $\mathcal{C}^\bot$);
    \item For positive integers $p$ and $q$, $J_{p \times q}$ denotes the all-one matrix of dimension $p \times q$;
    \item For a finite set $S$, $\sharp S$ denotes the number of elements in $S$;
    \item $\mathrm{rk}(A)$ denotes the rank of the matrix $A$;
    \item Let $\alpha \in \mathbb{F}_{2^m}$, the finite field with $2^m$ elements. Then $\mathrm{Tr}_1^m(\alpha)$ is the absolute trace of $\alpha$, that is, $\mathrm{Tr}_1^m(\alpha)=\sum_{i=0}^{m-1}\alpha^{2^i}$;
    \item For non-negative integers $a \leq b$, $[a \isep b]$ denotes the set of integers between $a$ and $b$ (both inclusive);
        \item $\Omega_p$ and $\widetilde{\Omega}_p$ are the uniform probability spaces on the set of all mappings $\{s: [1\isep p] \to \psi(\mathcal{C})\}$ and $\{\tilde{s}: [1\isep p] \to \phi(\mathcal{C})\}$ respectively, where $\psi: \mathbb{F}_2 \to \mathbb{R}$ and $\phi: \mathbb{F}_2 \to \mathbb{R}$ are defined by $\psi(a)=(-1)^a$ and $\phi(a)=a$ respectively for $a \in \{0,1\}$, extended component-wise to $\mathcal{C}$;
    \item For $s \in \Omega_p$ and $\tilde{s} \in \widetilde{\Omega}_p$, $\Phi_s$ and $\widetilde{\Phi}_{\tilde{s}}$ are $p \times n$ random matrices whose $i$-th rows are $s(i)$ and $\tilde{s}(i)$ respectively, and $\mathcal{G}_s$ and $\widetilde{\mathcal{G}}_{\tilde{s}}$ are the Gram matrices of $\frac{1}{\sqrt{n}}\Phi_s$ and $\frac{2}{\sqrt{n}}\widetilde{\Phi}_{\tilde{s}}$ respectively;
    \item $\lambda_1 \geq \lambda_2 \geq \cdots \geq \lambda_p$ and $\tilde{\lambda}_1 \geq \tilde{\lambda}_2 \geq \cdots \geq \tilde{\lambda}_p$ are the eigenvalues of $\mathcal{G}_s$ and $\widetilde{\mathcal{G}}_{\tilde{s}}$ respectively;
    \item $\mu_{\mathcal{G}_s}=\frac{1}{p}\sum_{i=1}^p \delta_{\lambda_i}$ and $\mu_{\tilde{\mathcal{G}}_{\tilde{s}}}=\frac{1}{p}\sum_{i=1}^p \delta_{\tilde{\lambda}_i}$ are the spectral measures of $\mathcal{G}_s$ and $\widetilde{\mathcal{G}}_{\tilde{s}}$ respectively, where $\delta_\lambda$ is the Dirac measure on $\lambda$;
    \item For $w \in [0\isep n]$, $A_w(\mathcal{C})$ and $B_w(\mathcal{C})=A_w(\mathcal{C}^\bot)$ denote the number of codewords of weight $w$ in $\mathcal{C}$ and $\mathcal{C}^\bot$ respectively;
    \item For any probability space $\Omega$, $\E(\cdot,\Omega), \mathbb{V}(\cdot,\Omega)$ and $\P(\cdot,\Omega)$ denote the expectation, variance and probability with respect to $\Omega$ respectively, and $\mathrm{Cov}(X,Y,\Omega)$ denotes the covariance between the two random variables $X$ and $Y$ with respect to $\Omega$;
        \item We write $f(n)=O(g(n))$ or $f(n)\ll g(n)$ to mean that there exists a positive constant $C$ independent of $n$ such that $\left|\frac{f(n)}{g(n)}\right| \leq C$ for all $n$; we write $f(n)=O_\l(g(n))$ or $f(n) \ll_\l g(n)$ to emphasize that the constant $C$ above depends on $\l$; if $f(n)$ (and possibly $g(n)$) is a function on a random variable defined on some probability space $\Omega(n)$, we write $f(n)=O_P(g(n))$ to mean that the said inequality holds with probability (with respect to $\Omega(n)$) converging to 1;
    \item We write $f(n)=o(g(n))$ to mean $\lim_{n \to \infty} \frac{f(n)}{g(n)}=0$. We write $f(n)=o_\l(g(n))$ to emphasize that this zero limit is under the assumption that $\l$ is independent of $n$; if $f(n)$ (and possibly $g(n)$) is a function on a random variable defined on some probability space $\Omega(n)$, we write $f(n)=o_P(g(n))$  to mean that the said convergence holds with probability (with respect to $\Omega(n)$) converging to 1;
    \item We write $f(n) \asymp g(n)$ to mean both $f(n)=O(g(n))$ and $g(n)=O(f(n))$, which is also equivalent to saying that there are positive constants $C_1 \leq C_2$ such that $C_1|g(n)| \leq |f(n)| \leq C_2|g(n)|$ for all $n$;
     \item Given sequences of random variables $\{X_n\}$ and $
    \{Y_n\}$, where $X_n$ and $Y_n$ are both over the same probability space $\Omega(n)$, we say that $X_n$ is \emph{stochastically dominated} by $Y_n$ (see \cite{Knowles, Bloemendal}), written symbolically as $X_n \prec Y_n$, if, for any $\ep > 0$ and $D > 0$, there exists $N$ depending only on $\ep,D$, such that for all $n > N$,
    $$\P(|X_n| > n^\ep|Y_n|,\Omega(n)) \leq n^{-D}.$$
    Moreover, for a deterministic number $c$, we write $X_n=c+O_\prec(Y_n)$ to mean $X_n-c \prec Y_n$.
    \item Let $\{Y_n\}$ be a sequence of random variables (not necessarily belonging to the same probability space) and $\mathcal{F}$ be a (fixed) probability distribution. We write $Y_n \xrightarrow{d} \mathcal{F}$ to mean that $Y_n$ converges in distribution to $\mathcal{F}$ as $n \to \infty$.

\end{itemize}
\subsection{Useful Results}
\begin{theorem}\cite[Theorem 1]{CTX}\label{mainthm0}
    Let $\mathcal{C}$ be a binary linear code of length $n$ and $d^\bot \geq 5$. Let $\Phi_s$ be the $p \times n$ $\pm 1$ random matrix as constructed above, with $y_n=p/n \to y \in (0,1)$. Let $\mathcal{G}_s$ be the Gram matrix of $\frac{1}{\sqrt{n}}\Phi_s$, and $\mu_{\mathcal{G}_s}$ denote the spectral measure of $\mathcal{G}_s$. Then as $n \to \infty$,
    \begin{equation}\label{Mu}
    |\mu_{\mathcal{G}_s}(\mathbf{I})-\varrho_{\mathrm{MP}_{y_n}}(\mathbf{I})| \prec n^{-1/4}
    \end{equation}
    uniformly for all intervals $\mathbf{I} \subset \mathbb{R}$.
\end{theorem}
Here $\varrho_{\mathrm{MP}_{y_n}}$ is the Marchenko-Pastur (MP) measure with parameter $y_n$, defined by
$$\varrho_{\mathrm{MP}_{y_n}}(dx)=\frac{1}{2\pi xy_n}\sqrt{(x-a)(b-x)}{\mathbbm{1}_{[a,b]}}dx,$$
where $a=(1-\sqrt{y_n})^2$ and $b=(1+\sqrt{y_n})^2$.
\begin{lemma}[Eigenvalue Interlacing]\label{EI}
    Let $A$ and $B$ be $p \times n$ real matrices with $\mathrm{rk}(B)=1$ and $p \leq n$. Let $G=AA^T$ and $G'=(A+B)(A+B)^T$ be the Gram matrices of $A$ and $A+B$ respectively. If $\lambda_1 \geq \lambda_2 \geq \cdots \geq \lambda_p$ and $\lambda_1' \geq \lambda_2' \geq \cdots \geq \lambda_p'$ are the eigenvalues of $G$ and $G'$ respectively, then for $i \in [1\isep p]$,
    $$\lambda_{i+1} \leq \lambda_i' \leq \lambda_{i-1}$$
    where by convention $\lambda_{p+1}=0$ and $\lambda_0=+\infty$.
\end{lemma}
\begin{proof}
    Note that the eigenvalues of $G$ and $G'$ are simply the squares of the singular values of $A$ and $A'=A+B$ respectively. Moreover, the nonzero singular values (counted with multiplicity) of $A$ and $A'$ are also the same as those of the $n \times n$ matrices
    $$A_0=\begin{bmatrix}
        A\\
        O_{(n-p) \times n}
    \end{bmatrix}$$
    and
    $$A_0'=\begin{bmatrix}
        A'\\
        O_{(n-p) \times n}
    \end{bmatrix}$$
    respectively, where $O_{(n-p) \times n}$ is the $(n-p) \times n$ zero matrix. Then the desired result follows from \cite[Theorem 1]{Thompson}.
\end{proof}
    
\begin{lemma}\label{W}
    Let $\mathcal{C}$ be a binary linear code with length $n$ and minimum distance $d$. As $n \to \infty$ while $d$ is fixed, for any fixed positive integer $w \geq d$, we have
    $$A_w(\mathcal{C})=O_w(n^{w-r}),$$
    where $r:=\lfloor \frac{d-1}{2}\rfloor$.
\end{lemma}
\begin{proof}
    By \cite[Eq. (5)]{Levy}, we have
    $$A_w(\mathcal{C}) \leq \min_v \frac{\binom{n}{w+v}}{N_n(d,w,v)},$$
    where the minimum is taken over integers $v$ with $|v| \leq r$,
    and
    $$N_n(d,w,v)=\sum_i \binom{w}{i}\binom{n-w}{v+i},$$
    with summation over all integers $i$ satisfying $2i+v \leq r$.

    Note that the summand of $N_n(d,w,v)$ is non-negative for all $v$ with $|v| \leq r$ and $i$ satisfying $2i+v \leq r$, we may simply take $v=-r$ and $i=r$ to get that
    $$A_w(\mathcal{C}) \leq \frac{\binom{n}{w-r}}{N_n(d,w,-r)} \leq \frac{\binom{n}{w-r}}{\binom{w}{r}}\ll_w n^{w-r}$$
    as desired.

\end{proof}

\section{Empirical Spectral Distribution of $\widetilde{\mathcal{G}}_{\tilde{s}}$}\label{Grelate}
In this section we aim at proving the following result regarding the convergence rate of the empirical spectral distribution of $\widetilde{\mathcal{G}}_{\tilde{s}}$.
\begin{theorem}\label{mainthm}
    Let $\mathcal{C}$ be a binary linear code of length $n$ and $d^\bot \geq 5$. Let $\tilde{s} \in \widetilde{\Omega}_p$ and  $\widetilde{\Phi}_{\tilde{s}}$ be the corresponding $p \times n$ 0-1 random matrix, with $y_n:=p/n \to y \in (0,1)$. Let $\widetilde{\mathcal{G}}_{\tilde{s}}$ be the Gram matrix of $\frac{2}{\sqrt{n}}\widetilde{\Phi}_{\tilde{s}}$, and $\mu_{\widetilde{\mathcal{G}}_{\tilde{s}}}$ denote the spectral measure of $\widetilde{\mathcal{G}}_{\tilde{s}}$. Then as $n \to \infty$,
        $$|\mu_{\widetilde{\mathcal{G}}_{\tilde{s}}}(\mathbf{I})-\varrho_{\mathrm{MP}_{y_n}}(\mathbf{I})| \prec n^{-1/4}$$
        uniformly for all intervals $\mathbf{I} \subset \mathbb{R}$.
\end{theorem}
\begin{proof}
If $s \in \Omega_p$ and $\tilde{s} \in \widetilde{\Omega}_p$ are derived from exactly the same realization of $p$ (not necessarily distinct) codewords from $\mathcal{C}$, then it can be easily verified that
\begin{equation}\label{E0}
\widetilde{\Phi}_{\tilde{s}}=\frac{1}{2}(J_{p \times n}-\Phi_s).
\end{equation}
Hence
\begin{equation}\label{E1}
\frac{2}{\sqrt{n}}\widetilde{\Phi}_{\tilde{s}}=\frac{1}{\sqrt{n}}J_{p \times n}-\frac{1}{\sqrt{n}}\Phi_s.
\end{equation}
Since $\mathrm{rk}\left(\frac{1}{\sqrt{n}}J_{p \times n}\right)=1$, Lemma \ref{EI} implies that, the eigenvalues of $\widetilde{\mathcal{G}}_{\tilde{s}}$ and $\mathcal{G}_s$ satisfy
    $$\lambda_{i+1} \leq \tilde{\lambda}_i \leq \lambda_{i-1}$$
    for $i \in [1\isep p]$, where for convention, $\lambda_0=+\infty$ and $\lambda_{p+1}=0$.

One direct consequence of this is that,
\begin{equation}\label{EI2}
|\mu_{\widetilde{\mathcal{G}}_{\tilde{s}}}(\mathbf{I})-\mu_{\mathcal{G}_s}(\mathbf{I})| \leq \frac{2}{p}
\end{equation}
uniformly for all intervals $\mathbf{I} \subset \mathbb{R}$.

Combining with Theorem \ref{mainthm0} completes the proof of Theorem \ref{mainthm}.
\end{proof}
\begin{remark}\label{Rk1}
    When $\mathcal{C}$ is the full binary code $\mathbb{F}_2^n$, the input matrices $\Phi_s$ and $\widetilde{\Phi}_{\tilde{s}}$ have i.i.d. entries. Results from \cite{Bloemendal, Pillai} show that the true convergence rate for $\mathcal{G}_s$ (central sample covariance matrix) is $O_\prec(n^{-1})$. This will then also be the case of $\widetilde{\mathcal{G}}_{\tilde{s}}$ (non-central sample covariance matrix) by the consequence of eigenvalue interlacing in (\ref{EI2}).
\end{remark}


\section{Proof of Theorem \ref{mainthm2}: The Limiting Behaviour of $\tilde{\lambda}_1$}\label{pthm2}
In this section we investigate the limiting behaviour of the largest eigenvalue $\tilde{\lambda}_1$ of $\widetilde{\mathcal{G}}_{\tilde{s}}$.

We actually aim to prove the following version of Theorem \ref{mainthm2}, which covers a larger class of codes.

\begin{theorem}\label{mainthm2'}
    Let $\mathcal{C}$ be a binary linear code of length $n$. Let $\widetilde{\Phi}_{\tilde{s}}$ be the $p \times n$ 0-1 random matrix as constructed above, with $y_n:=p/n \to y \in (0,1)$. Let $\widetilde{\mathcal{G}}_{\tilde{s}}$ be the Gram matrix of $\frac{2}{\sqrt{n}}\widetilde{\Phi}_{\tilde{s}}$, and $\tilde{\lambda}_1$ denote the largest eigenvalue of $\widetilde{\mathcal{G}}_{\tilde{s}}$. If $d^\bot \geq 3$ and $B_4(\mathcal{C})=o(n^3)$, then as $n \to \infty, \tilde{\lambda}_1-(p+1+y_n)$ converges in distribution to the Gaussian distribution with mean $0$ and variance $4y$.
\end{theorem}
The proof of Theorem \ref{mainthm2'} is divided into two parts as follows. Due to their technical nature, we assume them in this section to prove Theorem \ref{mainthm2'}. Their proofs are postponed to the next section.

\begin{lemma}\label{lemU}
    Let $\mathcal{C}$ be a binary linear code of length $n$ with $d^\bot \geq 3$. Let $\mathbf{u}=(u_1,u_2,\cdots,u_p)^T$, where $u_i=\frac{2}{n}\mathrm{wt}(\mathbf{c}_i)$ and $\mathbf{c}_i$ is the codeword corresponding to $\tilde{s}(i)$ for $1 \leq i \leq p$. Assume $y_n=p/n \to y \in (0,1)$. Then 
    \begin{enumerate}
        \item $\E(\|\mathbf{u}\|^2,\widetilde{\Omega}_p)=p+y_n$ and $\mathbb{V}(\|\mathbf{u}\|^2,\widetilde{\Omega}_p)=O(1)$;
        \item if we further assume $B_4(\mathcal{C})=o(n^3)$, then $\mathbb{V}(\|\mathbf{u}\|^2,\widetilde{\Omega}_p)=4y_n+o(1)$ and
        $$\|\mathbf{u}\|^2-(p+y_n) \xrightarrow{d} \mathcal{N}(0,4y).$$
    \end{enumerate}

\end{lemma}

\begin{lemma}\label{lemLUS}
   Let $\mathcal{C}$ be a binary linear code of length $n$ with $d^\bot \geq 3$. Let $\widetilde{\Phi}_{\tilde{s}}$ be the $p \times n$ 0-1 random matrix as constructed above, with $y_n:=p/n \to y \in (0,1)$. Let $\widetilde{\mathcal{G}}_{\tilde{s}}$ be the Gram matrix of $\frac{2}{\sqrt{n}}\widetilde{\Phi}_{\tilde{s}}$, and $\tilde{\lambda}_1$ denotes its largest eigenvalue. Moreover, let $\mathbf{u}=(u_1,u_2,\cdots,u_p)^T$, where $u_i=\frac{2}{n}\mathrm{wt}(\mathbf{c}_i)$ and $\mathbf{c}_i$ is the codeword corresponding to $\tilde{s}(i)$ for $1 \leq i \leq p$. Then as $n \to \infty$, we have
  $$\P(|\tilde{\lambda}_1-(\|\mathbf{u}\|^2+1)| > n^{-1/4},\widetilde{\Omega}_p)=O(n^{-1/2}).$$
\end{lemma}

\begin{proof}[Proof of Theorem \ref{mainthm2'}]
    By Statement 2 of Lemma \ref{lemU}, we know that, for any $x \in \mathbb{R}$,
$$\P(\|\mathbf{u}\|^2-(p+y_n) \leq x,\widetilde{\Omega}_p)=F(x)+o(1)$$
where
$$F(x)=\frac{1}{\sqrt{8\pi y}}\int_{-\infty}^x e^{-\frac{t^2}{8y}}dt$$
is the cumulative distribution function (CDF) of $\mathcal{N}(0,4y)$.

Then, on one hand, we have
\begin{align*}
    \P(\tilde{\lambda}_1-(p+1+y_n) \leq x,\widetilde{\Omega}_p)&=\P(\tilde{\lambda}_1-(p+1+y_n) \leq x \wedge |\tilde{\lambda}_1-(\|\mathbf{u}\|^2+1)| \leq  n^{-1/4},\widetilde{\Omega}_p)\\
    &+\P(\tilde{\lambda}_1-(p+1+y_n) \leq x \wedge |\tilde{\lambda}_1-(\|\mathbf{u}\|^2+1)| > n^{-1/4},\widetilde{\Omega}_p)\\
    &\leq \P(\|\mathbf{u}\|^2-(p+y_n) \leq x+n^{-1/4},\widetilde{\Omega}_p)+\P(|\tilde{\lambda}_1-(\|\mathbf{u}\|^2+1)| > n^{-1/4},\widetilde{\Omega}_p)\\
    &=(F(x+n^{-1/4})+o(1))+O(n^{-1/2})\\
    &=F(x)+o(1).
\end{align*}
The last asymptotic equality follows from the continuity of $F$.

On the other hand,
\begin{align}
    &\,\P(\tilde{\lambda}_1-(p+1+y_n) \leq x,\widetilde{\Omega}_p)\nonumber\\
    &=1-\P(\tilde{\lambda}_1-(p+1+y_n) > x,\widetilde{\Omega}_p)\nonumber\\
    &=1-\P(\tilde{\lambda}_1-(p+1+y_n) > x \wedge |\tilde{\lambda}_1-(\|\mathbf{u}\|^2+1)| \leq  n^{-1/4},\widetilde{\Omega}_p)\nonumber\\
    &-\P(\tilde{\lambda}_1-(p+1+y_n) > x \wedge |\tilde{\lambda}_1-(\|\mathbf{u}\|^2+1)| >  n^{-1/4},\widetilde{\Omega}_p)\nonumber\\
    &\geq 1-\P(\|\mathbf{u}\|^2-(p+y_n) > x-n^{-1/4},\widetilde{\Omega}_p)-\P(|\tilde{\lambda}_1-(\|\mathbf{u}\|^2+1)| >  n^{-1/4},\widetilde{\Omega}_p)\nonumber\\
    &=1-(1-F(x-n^{-1/4})+o(1))-O(n^{-1/2})\nonumber\\
    &=F(x)+o(1)\label{CDF}.
\end{align}
Combining both asymptotic inequalities completes the proof of Theorem \ref{mainthm2'}.
\end{proof}
\section{Proof of Lemmas \ref{lemU} and \ref{lemLUS}}\label{plemLUS}
The proofs of Lemmas \ref{lemU} and \ref{lemLUS} involve technical complexity. To streamline the idea of proof, we first assume the following results and prove Lemmas \ref{lemU} and \ref{lemLUS} from them. Note that Lemmas \ref{lambda} and \ref{lemV} will also be needed in the proof of Lemma \ref{lemG} later.
\begin{lemma}\label{lambda}
    Assume $d^\bot \geq 3$. Then $\lambda_1 \prec 1$.
\end{lemma}

\begin{lemma}\label{lemV}
    Assume $d^\bot \geq 3$. Define $\mathbf{v}:=\mathbf{1}_p-\mathbf{u}$. Then $\E(\mathbf{v},\widetilde{\Omega}_p)=\mathbf{0}, \E(\|\mathbf{v}\|^2,\widetilde{\Omega}_p)=y_n$ and $\mathbb{V}(\|\mathbf{v}\|^2,\widetilde{\Omega}_p)=O(1)$. Moreover, if $B_4(\mathcal{C})=o(n^3)$, then $\E(v_i^4,\widetilde{\Omega}_p)=o(n^{-1})$ and $\mathbb{V}(\|\mathbf{v}\|^2,\widetilde{\Omega}_p)=o(1)$.
\end{lemma}

\begin{lemma}\label{lemG}
    Write $\widehat{\mathcal{G}}_s=\mathcal{G}_s-\mathbf{v}\mathbf{v}^T$, where $\mathbf{v}=\mathbf{1}_p-\mathbf{u}$. If $d^\bot \geq 3$, then for any constant $C > 0$, $$\P(|\mathbf{u}^T\widehat{\mathcal{G}}_s\mathbf{u}-p| > Cn^{3/4},\widetilde{\Omega}_p) = O(n^{-1/2}).$$
\end{lemma}
\begin{proof}[Proof of Lemma \ref{lemU}]
\begin{enumerate}
    \item We first write
    \begin{align}
        \|\mathbf{u}\|^2&=\sum_{i=1}^p u_i^2\nonumber\\
        &=\sum_{i=1}^p (1-v_i)^2\nonumber\\
        &=p-2\sum_{i=1}^p v_i+\|\mathbf{v}\|^2,\label{U}
        \end{align}
        where $\mathbf{v}=\mathbf{1}_p-\mathbf{u}$.
        
        Hence by Lemma \ref{lemV} we have
        $$\E(\|\mathbf{u}\|^2,\widetilde{\Omega}_p)=p+y_n.$$
        To find the variance, we first define
        \begin{equation}\label{Z0}
        Z:=-2\sum_{i=1}^p v_i,
        \end{equation}
        so that
        \begin{equation}\label{U2}
            \|\mathbf{u}\|^2=p+\|\mathbf{v}\|^2+Z.
        \end{equation}
        Since $v_i$'s are independent, we have
        \begin{align}
        \mathbb{V}(Z,\widetilde{\Omega}_p)&=4\sum_{i=1}^p \mathbb{V}(v_i,\widetilde{\Omega}_p)\nonumber\\
        &=4\E(\|\mathbf{v}\|^2,\widetilde{\Omega}_p)\nonumber\\
        &=4y_n\label{Z}
        \end{align}
        by Lemma \ref{lemV}.

        Hence, by Cauchy-Schwarz inequality, we have
        \begin{align}
            \mathrm{Cov}(\|\mathbf{v}\|^2,Z,\widetilde{\Omega}_p)&=\E[(\|\mathbf{v}\|^2-\E(\|\mathbf{v}\|^2,\widetilde{\Omega}_p))(Z-\E(Z,\widetilde{\Omega}_p)),\widetilde{\Omega}_p]\nonumber\\
            &\ll (\E[(\|\mathbf{v}\|^2-\E(\|\mathbf{v}\|^2,\widetilde{\Omega}_p))^2,\widetilde{\Omega}_p])^{1/2}(\E[(Z-\E(Z,\widetilde{\Omega}_p))^2,\widetilde{\Omega}_p])^{1/2}\nonumber\\
            &=(\mathbb{V}(\|\mathbf{v}\|^2,\widetilde{\Omega}_p))^{1/2}(\mathbb{V}(Z,\widetilde{\Omega}_p))^{1/2}\label{Cov}\\
            &\ll 1\nonumber.
        \end{align}
        and therefore by (\ref{U2}), (\ref{Z}) and Lemma \ref{lemV},
        \begin{align*}
        \mathbb{V}(\|\mathbf{u}\|^2,\widetilde{\Omega}_p)&=\mathbb{V}(\|\mathbf{v}\|^2,\widetilde{\Omega}_p)+\mathbb{V}(Z,\widetilde{\Omega}_p)+2\mathrm{Cov}(\|\mathbf{v}\|^2,Z,\widetilde{\Omega}_p)\\
        &\ll 1.
        \end{align*}
        \item We substitute the refined bound in Lemma \ref{lemV} to (\ref{Cov}) to obtain
        $$\mathrm{Cov}(\|\mathbf{v}\|^2,Z,\widetilde{\Omega}_p)=o(1).$$
        Thus by (\ref{U2}), (\ref{Z}) and Lemma \ref{lemV},
        \begin{align*}
        \mathbb{V}(\|\mathbf{u}\|^2,\widetilde{\Omega}_p)&=\mathbb{V}(\|\mathbf{v}\|^2,\widetilde{\Omega}_p)+\mathbb{V}(Z,\widetilde{\Omega}_p)+2\mathrm{Cov}(\|\mathbf{v}\|^2,Z,\widetilde{\Omega}_p)\\
        &=4y_n+o(1).
        \end{align*}
        Finally, we have
        $$\frac{\sum_{i=1}^p\E( (-2v_i)^4,\widetilde{\Omega}_p)}{(\mathbb{V}(Z,\widetilde{\Omega}_p))^2}=\frac{\sum_{i=1}^p\E(v_i^4,\widetilde{\Omega}_p)}{y_n^2}=o(1)$$
        by Lemma \ref{lemV},
        so Lyapunov's Central Limit Theorem with $\delta=2$ \cite[Theorem 27.3]{Lyapunov} yields
        $$Z \xrightarrow{d} \mathcal{N}(0,4y).$$
        Then the desired convergence on $\|\mathbf{u}\|^2$ follows from (\ref{U2}) and Lemma \ref{lemV}.

        This completes the proof of Lemma \ref{lemU}.
        
\end{enumerate}
\end{proof}
\begin{proof}[Proof of Lemma \ref{lemLUS}]
    By (\ref{E1}), we can write
    \begin{align*}
        \widetilde{\mathcal{G}}_{\tilde{s}}&=\left(\frac{1}{\sqrt{n}}J_{p \times n}-\frac{1}{\sqrt{n}}\Phi_s\right)\left(\frac{1}{\sqrt{n}}J_{p \times n}-\frac{1}{\sqrt{n}}\Phi_s\right)^T\\
        &=J_{p \times p}-\frac{1}{n}J_{p \times n}\Phi_s^T-\frac{1}{n}\Phi_sJ_{n \times p}+\mathcal{G}_s\\
        &=\mathbf{1}_p\mathbf{1}_p^T-\mathbf{1}_p\mathbf{v}^T-\mathbf{v}\mathbf{1}_p^T+\mathcal{G}_s\\
        &=\mathcal{G}_s+\mathbf{u}\mathbf{u}^T-\mathbf{v}\mathbf{v}^T\\
        &=\widehat{\mathcal{G}}_s+\mathbf{u}\mathbf{u}^T.
\end{align*}
By \cite{Stewart},
$$|\tilde{\lambda}_1-\|\mathbf{u}\|^2| \leq \|\widehat{\mathcal{G}}_s\| \leq \lambda_1+\|\mathbf{v}\|^2.$$
By Lemmas \ref{lambda} and \ref{lemV} with Chebyshev's inequality, we have
$$\lambda_1+\|\mathbf{v}\|^2=O(n^{1/4})$$
with probability at least $1-O(n^{-1/2})$,
while, by Lemma \ref{lemU} with Chebyshev's inequality,
$$\|\mathbf{u}\|^2-p=O(n^{1/4})$$
with probability at least $1-O(n^{-1/2})$.

Combining both, we have
\begin{equation}\label{Lp}
\tilde{\lambda}_1 -p = O(n^{1/4})
\end{equation}
with probability at least $1-O(n^{-1/2})$.

This implies $\tilde{\lambda}_1$ is much larger than $\|\widehat{\mathcal{G}}_s\|$ with high probability, and hence $\det(\widehat{\mathcal{G}}_s-\tilde{\lambda}_1 I_p) \neq 0$ with high probability.

Thereby we may write
\begin{align*}
0&=\det(\widetilde{\mathcal{G}}_{\tilde{s}}-\tilde{\lambda}_1 I_p)\\
&=\det((\widehat{\mathcal{G}}_s-\tilde{\lambda}_1 I_p)+\mathbf{u}\mathbf{u}^T)\\
&=\det(\widehat{\mathcal{G}}_s-\tilde{\lambda}_1 I_p)\det(I_p+\mathbf{u}\mathbf{u}^T(\widehat{\mathcal{G}}_s-\tilde{\lambda}_1 I_p)^{-1}),
\end{align*}
which implies
$$1=\mathbf{u}^T(\tilde{\lambda}_1 I_p-\widehat{\mathcal{G}}_s)^{-1}\mathbf{u}.$$
Assume $\tilde{\lambda}_1 \neq 0$. Then we may express
$$(\tilde{\lambda}_1 I_p-\widehat{\mathcal{G}}_s)^{-1}=\tilde{\lambda}_1^{-1}I_p+\sum_{i=1}^\infty\tilde{\lambda}_1^{-(i+1)}\widehat{\mathcal{G}}_s^i,$$
so that
\begin{equation}\label{SUG}
\tilde{\lambda}_1=\|\mathbf{u}\|^2+\sum_{i=1}^\infty \frac{\mathbf{u}^T \widehat{\mathcal{G}}_s^i\mathbf{u}}{\tilde{\lambda}_1^i}.
\end{equation}
Note that
$$\sum_{i=2}^\infty\frac{\mathbf{u}^T \widehat{\mathcal{G}}_s^i\mathbf{u}}{\tilde{\lambda}_1^i}\ll \sum_{i=2}^\infty\frac{\|\mathbf{u}\|^2\|\widehat{\mathcal{G}}_s\|^i}{\tilde{\lambda}_1^i} \ll \sum_{i=2}^\infty n^{1-3i/4} \ll n^{-1/2}$$
with probability at least $1-O(n^{-1/2})$.

Moreover, by Lemma \ref{lemG} and (\ref{Lp}), we have, for some constant $C > 0$,
\begin{align*}
&\,\P\left(\left|\frac{\mathbf{u}^T\widehat{\mathcal{G}}_s\mathbf{u}}{\tilde{\lambda}_1}-1\right| > \frac{1}{2}n^{-1/4} \wedge \tilde{\lambda}_1 \neq 0,\widetilde{\Omega}_p\right)\\
&\leq\P\left(|\mathbf{u}^T\widehat{\mathcal{G}}_s\mathbf{u}-\tilde{\lambda}_1| > \frac{1}{2}n^{-1/4}(p-Cn^{1/4}),\widetilde{\Omega}_p\right)+\P(\tilde{\lambda}_1 < p-Cn^{1/4},\widetilde{\Omega}_p)\\
&\leq \P\left(|\mathbf{u}^T\widehat{\mathcal{G}}_s\mathbf{u}-p| > \frac{y}{8}n^{3/4},\widetilde{\Omega}_p\right)+\P\left(|\tilde{\lambda}_1-p| > \frac{y}{8}n^{3/4},\widetilde{\Omega}_p\right)+\P(|\tilde{\lambda}_1-p| > Cn^{1/4},\widetilde{\Omega}_p)\\
&=O(n^{-1/2}).
\end{align*}
Hence by (\ref{SUG}),
\begin{equation}\label{UL}
    \P(|\tilde{\lambda}_1-(\|\mathbf{u}\|^2+1)| > n^{-1/4} \wedge \tilde{\lambda}_1 \neq 0,\widetilde{\Omega}_p)=O(n^{-1/2}).
\end{equation}
When $\tilde{\lambda}_1=0$, that is, $\tilde{s}(i)=\mathbf{0}$ for all $i$, we obviously have $\mathbf{u}=\mathbf{0}$, so that
$$|\tilde{\lambda}_1-\|\mathbf{u}\|^2-1|=1 > n^{-1/4}.$$
But then we have
$$\P(\tilde{\lambda}_1=0,\widetilde{\Omega}_p)=\left(\frac{1}{\sharp \mathcal{C}}\right)^p.$$
Since $d^\bot \geq 3$, by the sphere-packing bound \cite[Chapter 17, Theorem 12]{MacWilliams} with $\delta=1$, we obtain
$$\sharp \mathcal{C}^\bot\left(1+\binom{n}{1}\right) \leq 2^n,$$
which implies
$$\sharp \mathcal{C}=\frac{2^n}{\sharp \mathcal{C}^\bot} \geq 1+\binom{n}{1}=1+n.$$
That is,
$$\P(\tilde{\lambda}_1=0,\widetilde{\Omega}_p) \leq \left(\frac{1}{1+n}\right)^p.$$
Combining with (\ref{UL}), we obtain
$$\P(|\tilde{\lambda}_1-(\|\mathbf{u}\|^2+1)| > n^{-1/4})=O(n^{-1/2}).$$
This completes the proof of Lemma \ref{lemLUS}.

\end{proof}

\begin{remark}\label{Rk2}
If $\mathcal{C}=\mathbb{F}_2^n$, then the entries $\tilde{s}(i)[t]$ of $\widetilde{\Phi}_{\tilde{s}}$ are i.i.d. random variables with mean $\mu=\frac{1}{2}$ and variance $\sigma^2=\frac{1}{4}$.

    Together with (\ref{E0}), applying \cite[Proposition 3.1]{Bryc} (or equivalently \cite[Theorem 1.3]{Silverstein}) implies that the largest singular value of $\widetilde{\Phi}_{\tilde{s}}$ satisfies
    \begin{equation}\label{sigma}
    \tilde{\sigma}_1=\frac{\sqrt{pn}}{2}+\frac{1}{4}\left(\sqrt{y_n}+\frac{1}{\sqrt{y_n}}\right)+Z_1+\ep^{(n)}
    \end{equation}
    where
    $$Z_1=-\frac{1}{2\sqrt{pn}}\sum_{i=1}^p\sum_{t=1}^n s(i)[t]$$
    and $\ep^{(n)}=o_P(1)$.
    
    Note that
    $$\sum_{t=1}^n s(i)[t]=\sum_{t=1}^n (1-2\tilde{s}(i)[t])=n-2\mathrm{wt}(\mathbf{c}_i)=nv_i.$$
    Hence
    $$Z_1=-\frac{1}{2\sqrt{y_n}}\sum_{i=1}^p v_i=\frac{1}{4\sqrt{y_n}}Z.$$
    Then
    \begin{align}
        \tilde{\lambda}_1&=\frac{4}{n}\tilde{\sigma}_1^2\nonumber\\
        &=\frac{4}{n}\left(\frac{\sqrt{pn}}{2}+\frac{1}{4}\left(\sqrt{y_n}+\frac{1}{\sqrt{y_n}}\right)+\frac{1}{4\sqrt{y_n}}Z+\ep^{(n)}\right)^2\nonumber\\
        &=\frac{4}{n}\left(\frac{pn}{4}+\frac{\sqrt{pn}}{4}\left(\sqrt{y_n}+\frac{1}{\sqrt{y_n}}\right)+\frac{n}{4}Z+O(n\ep^{(n)}+(\ep^{(n)})^2)+O_P(1)\right)\nonumber\\
    &=p+y_n+1+Z+\ep^{'(n)}\label{L4},
    \end{align}
    where $\ep^{'(n)}=O(\ep^{(n)}+n^{-1}(\ep^{(n)})^2)+O_P(n^{-1})=o_P(1)$.
    
    On the other hand, if we only assume pairwise independence within row entries, which is equivalent to $d^\bot \geq 3$, by putting (\ref{U2}) into the statement of Lemma \ref{lemLUS}, we obtain
\begin{equation}\label{L3}    \tilde{\lambda}_1=p+1+\|\mathbf{v}\|^2+Z+O_P(n^{-1/4})=p+1+y_n+Z+\tilde{\ep}^{(n)}
\end{equation}
    where
$$\tilde{\ep}^{(n)}=\|\mathbf{v}\|^2-y_n+O_P(n^{-1/4}).$$
   The only difference of the equations (\ref{L4}) and (\ref{L3}) is in the remainder term $\ep^{'(n)}$ and $\tilde{\ep}^{(n)}$. Here we still need to control the difference between $\|\mathbf{v}\|^2$ and (its expectation) $y_n$ in order to get $\tilde{\ep}^{(n)}=o_P(1)$. Our result in Lemma \ref{lemU} implies that this is the case if we further assume $B_4(\mathcal{C})=o(n^3)$, and in particular under the stronger condition $d^\bot \geq 5$ ($4$-wise independence within row entries) in Theorem \ref{mainthm2}. We will see in Section \ref{SRM} that the condition on $B_4(\mathcal{C})$ cannot be dropped in general.
\end{remark}
\subsection{Proof of Lemma \ref{lambda}}
By \cite{OQBT}, we know that, for any fixed positive integer $\l$,
$$\E(A_\l(s),\Omega_p)=\beta_{\mathrm{MP}_{y_n},\l}+E_\l,$$
where $A_\l(s)$ is the $\l$-th moment of $\mu_{\mathcal{G}_s}$, $\beta_{\mathrm{MP}_{y_n},\l}$ is the $\l$-th moment of the MP law with parameter $y_n$, and
$$|E_\l| \leq \left(4+2\sqrt{2B_4(\mathcal{C})+\frac{1}{4}}\right)\frac{\l^{\l+1}}{n}.$$
Since $d^\bot \geq 3$, Lemma \ref{W} implies that $B_4(\mathcal{C})=O(n^3)$. Then we have
$$|E_\l| \ll_\l \frac{n^{3/2}}{n}\ll_\l n^{1/2}.$$
Hence
$$\E\left(\sum_{i=1}^p \lambda_i^\l,\Omega_p\right)=p\E(A_\l(s),\Omega_p)=O_\l(n^{3/2}).$$
Since all $\lambda_i$'s are non-negative, we must have
$$\E(\lambda_1^\l,\Omega_p) \leq \E\left(\sum_{i=1}^p \lambda_i^\l,\Omega_p\right)=O_\l(n^{3/2}).$$
Applying Markov's inequality, we see that, for any $\ep > 0$ and $D > 0$,
$$\P(\lambda_1 > n^{\ep},\Omega_p) \leq \frac{\E(\lambda_1^\l,\Omega_p)}{n^{\ep\l}} \leq C_\l n^{3/2-\ep\l} < n^{-D}$$
by choosing $\l=\lceil (D+5/2)/\ep\rceil$ and $n > C_\l$.

\subsection{Proof of Lemma \ref{lemV}}
For $i \in [1\isep p]$,
\begin{align*}
    \E(v_i,\widetilde{\Omega}_p)&=\E\left(1-\frac{2}{n}\mathrm{wt}(\mathbf{c}_i),\widetilde{\Omega}_p\right)\\
    &=-\frac{2}{n}\sum_{t=1}^n \E(\bar{c}_{it},\widetilde{\Omega}_p)\\
    &=0,
\end{align*}
where $\bar{c}_{it}:=\tilde{s}(i)[t]-\frac{1}{2}$ is uniform in $\{\pm \frac{1}{2}\}$ since $\tilde{s}(i)[t]$ is uniform in $\{0,1\}$ by $d^\bot \geq 3$. Collecting them all, we get $\E(\mathbf{v},\widetilde{\Omega}_p)=\mathbf{0}$.

Moreover,
\begin{align*}
    \E(v_i^2,\widetilde{\Omega}_p)&=\E\left(\left(1-\frac{2}{n}\mathrm{wt}(\mathbf{c}_i)\right)^2,\widetilde{\Omega}_p\right)\\
    &=\frac{4}{n^2}\sum_{t_1,t_2=1}^n \E(\bar{c}_{it_1}\bar{c}_{it_2},\widetilde{\Omega}_p).
\end{align*}
Since $d^\bot \geq 3$, $\bar{c}_{it_1}$ and $\bar{c}_{it_2}$ are independent for $t_1 \neq t_2$. In addition, $
\bar{c}_{it}^2$ is deterministically $\frac{1}{4}$. Thus
$$\E(v_i^2,\widetilde{\Omega}_p)=\frac{4}{n^2}\left(\frac{n}{4}\right)=\frac{1}{n}.$$
Hence $\E(\|\mathbf{v}\|^2,\widetilde{\Omega}_p)=\sum_{i=1}^p \E(v_i^2,\widetilde{\Omega}_p)=y_n$.

To estimate $\mathbb{V}(\|\mathbf{v}\|^2,\widetilde{\Omega}_p)$, we need to first compute $\E(v_i^4,\widetilde{\Omega}_p)$.

We have
\begin{align*}
\E(v_i^4,\widetilde{\Omega}_p)&=\E\left(\left(1-\frac{2}{n}\mathrm{wt}(\mathbf{c}_i)\right)^4,\widetilde{\Omega}_p\right)\\
&=\frac{16}{n^4}\sum_{t_1,t_2,t_3,t_4=1}^n\E( \bar{c}_{it_1}\bar{c}_{it_2}\bar{c}_{it_3}\bar{c}_{it_4},\widetilde{\Omega}_p).
\end{align*}
If $t_1,t_2,t_3,t_4$ come in pairs, this contributes at most $O(n^2)$ terms each equal $\frac{1}{16}$; otherwise some $t_j$ appears only once. Since $d^\bot \geq 3$, the only possible situation where this may contribute to the sum is when $t_1,t_2,t_3,t_4$ are all distinct. This corresponds to weight 4 codewords in $\mathcal{C}^\bot$ (terms involving a weight 3 codeword in $\mathcal{C}^\bot$ will leave one remaining $\bar{c}_{it_j}$ separated to take expectation on, eventually making them vanish). By Lemma \ref{W}, the number of such terms that are nonzero is $O(n^3)$.

Combining all, we then get $\E(v_i^4,\widetilde{\Omega}_p)=O(n^{-1})$, which implies
$$\mathbb{V}(v_i^2,\widetilde{\Omega}_p)=\E(v_i^4,\widetilde{\Omega}_p)-(\E(v_i^2,\widetilde{\Omega}_p))^2=O\left(\frac{1}{n}\right).$$
Since $v_i$'s are independent, we then obtain
$$\mathbb{V}(\|\mathbf{v}\|^2,\widetilde{\Omega}_p)=\sum_{i=1}^p \mathbb{V}(v_i^2,\widetilde{\Omega}_p)=O(1).$$
If we further assume $B_4(\mathcal{C})=o(n^3)$, then repeating the above arguments will yield
$\E(v_i^4,\widetilde{\Omega}_p)=o(n^{-1})$ and $\mathbb{V}(v_i^2,\widetilde{\Omega}_p)=o(n^{-1})$, and so $\mathbb{V}(\|\mathbf{v}\|^2,\widetilde{\Omega}_p)=o(1)$.

This completes the proof of Lemma \ref{lemV}.
\subsection{Proof of Lemma \ref{lemG}}

We have
$$\mathbf{u}^T\widehat{\mathcal{G}}_s\mathbf{u}=\mathbf{u}^T(\mathcal{G}_s-\mathbf{v}\mathbf{v}^T)\mathbf{u}=\mathbf{u}^T\mathcal{G}_s\mathbf{u}-(\mathbf{u}^T\mathbf{v})^2.$$
We hence split the required probability estimate into the following two estimates and prove them first.
\begin{align}
    \P\left(|\mathbf{u}^T\mathcal{G}_s\mathbf{u}-p| > \frac{3C}{4}n^{3/4},\widetilde{\Omega}_p\right)&=O(n^{-1/2}) \label{UG1}\\
    \P\left(|\mathbf{u}^T\mathbf{v}| > \frac{C^{1/2}}{2}n^{3/8},\widetilde{\Omega}_p\right)&=O(n^{-1/2}) \label{UG2}.
\end{align}
We first prove (\ref{UG2}).

We have
\begin{align*}
\E(\mathbf{u}^T\mathbf{v},\widetilde{\Omega}_p)&=\sum_{i=1}^p\E(u_iv_i,\widetilde{\Omega}_p)\\
&=\sum_{i=1}^p\E(v_i,\widetilde{\Omega}_p)-\sum_{i=1}^p\E(v_i^2,\widetilde{\Omega}_p)\\
&=-\frac{p}{n}\\
&=-y_n
\end{align*}
and
\begin{align*}
\mathbb{V}(\mathbf{u}^T\mathbf{v},\widetilde{\Omega}_p)&=\E((\mathbf{u}^T\mathbf{v})^2,\widetilde{\Omega}_p)-(\E(\mathbf{u}^T\mathbf{v},\widetilde{\Omega}_p))^2\\
&=\E\left(\left(\sum_{i=1}^p u_iv_i\right)^2,\widetilde{\Omega}_p\right)-y_n^2\\
&=\sum_{i=1}^p \E(u_i^2v_i^2,\widetilde{\Omega}_p)+\sum_{1 \leq i \neq j \leq p} \E(u_iv_i,\widetilde{\Omega}_p)\E(u_jv_j,\widetilde{\Omega}_p)-y_n^2\\
&\leq 4\sum_{i=1}^p \E(v_i^2,\widetilde{\Omega}_p)+\sum_{1 \leq i \neq j \leq p} (\E(v_i,\widetilde{\Omega}_p)-\E(v_i^2,\widetilde{\Omega}_p))(\E(v_j,\widetilde{\Omega}_p)-\E(v_j^2,\widetilde{\Omega}_p))-y_n^2\\
&=\frac{4p}{n}+\frac{p(p-1)}{n^2}-y_n^2\\
&\ll 1.
\end{align*}
In the first term of the fourth line we use the fact that $0 \leq u_i=\frac{2}{n}\mathrm{wt}(\mathbf{c}_i) \leq 2$.

By Chebyshev's inequality, then we have
$$\P\left(|\mathbf{u}^T\mathbf{v}| > \frac{C^{1/2}}{2}n^{3/8},\widetilde{\Omega}_p\right) \leq \P\left(|\mathbf{u}^T\mathbf{v}+y_n| > \frac{C^{1/2}}{4}n^{3/8},\widetilde{\Omega}_p\right) \ll n^{-3/4} \ll n^{-1/2}.$$
Now we prove (\ref{UG1}). We have
\begin{equation}\label{EV0}
\mathbf{u}^T\mathcal{G}_s\mathbf{u}=(\mathbf{1}_p-\mathbf{v})^T\mathcal{G}_s(\mathbf{1}-\mathbf{v})=\mathbf{1}_p^T\mathcal{G}_s\mathbf{1}_p-2\mathbf{v}^T\mathcal{G}_s\mathbf{1}_p+\mathbf{v}^T\mathcal{G}_s\mathbf{v}.
\end{equation}
We have, by Lemmas \ref{lambda} and \ref{lemV},
\begin{equation}\label{EV1}
|\mathbf{v}^T\mathcal{G}_s\mathbf{1}_p|\leq \lambda_1\|\mathbf{v}\|\|\mathbf{1}_p\| \leq \frac{C}{8}n^{1/8}n^{1/8}\sqrt{p} \leq \frac{C}{8}n^{3/4}
\end{equation}

with probability at least $1-O(n^{-1/2})$, and
\begin{equation}\label{EV2}
|\mathbf{v}^T\mathcal{G}_s\mathbf{v}| \leq \lambda_1\|\mathbf{v}\|^2 \leq \frac{C}{4}n^{1/8}n^{1/4}\leq\frac{C}{4}n^{3/4}
\end{equation}
with probability at least $1-O(n^{-1/2})$.

Moreover,
\begin{align*}
\mathbf{1}_p^T\mathcal{G}_s\mathbf{1}_p&=\sum_{i,j=1}^p\mathcal{G}_{ij}\\
&=\frac{1}{n}\sum_{i,j=1}^p \langle s(i),s(j)\rangle\\
&=\frac{1}{n}\sum_{i,j=1}^p\sum_{t=1}^n s(i)[t]s(j)[t]\\
&=\frac{1}{n}\sum_{t=1}^n c_t'^2,
\end{align*}
where $c_t':=\sum_{i=1}^p s(i)[t]$.

We have
\begin{align*}
\E(c_t'^2,\widetilde{\Omega}_p)&=\sum_{i,j=1}^p \E(s(i)[t]s(j)[t],\widetilde{\Omega}_p)\\
&=\sum_{i=1}^p\E((s(i)[t])^2,\widetilde{\Omega}_p)+\sum_{1 \leq i \neq j \leq p}\E(s(i)[t]s(j)[t],\widetilde{\Omega}_p)\\
&=p,
\end{align*}
so that
$$\E(\mathbf{1}_p^T\mathcal{G}_s\mathbf{1}_p,\widetilde{\Omega}_p)=p.$$
In addition,
\begin{align*}
    \mathbb{V}(c_t'^2,\widetilde{\Omega}_p)&=\E(c_t'^4,\widetilde{\Omega}_p)-p^2\\
    &=\sum_{i,j,i',j'=1}^p\E(s(i)[t]s(j)[t]s(i')[t]s(j')[t],\widetilde{\Omega}_p)-p^2\\
    &=O(n^2),
\end{align*}
since distinct rows of $\Phi_s$ are mutually independent.

Moreover, since $d^\bot \geq 3$, $c_{t_1}'$ and $c_{t_2}'$ are independent for $t_1 \neq t_2$.

Hence
\begin{align*}
    \mathbb{V}(\mathbf{1}_p^T\mathcal{G}_s\mathbf{1}_p,\widetilde{\Omega}_p)&=\frac{1}{n^2}\sum_{t=1}^n \mathbb{V}(c_t'^2,\widetilde{\Omega}_p)\\
    &=O(n).
\end{align*}
Therefore, by Chebyshev's inequality,
\begin{equation}\label{EV3}
\P\left(|\mathbf{1}_p^T\mathcal{G}_s\mathbf{1}_p-p| > \frac{C}{4}n^{3/4}\right) \ll \frac{n}{n^{3/2}} \ll n^{-1/2}.
\end{equation}
Combining (\ref{EV0})--(\ref{EV3}) yields (\ref{UG1}).

This completes the proof of Lemma \ref{lemG}.

\section{Simplex Codes $\mathscr{S}_m$ and First-Order Reed-Muller Codes $\mathrm{RM}(1,m)$: Counterexamples of Group Randomness}\label{SRM}
In this section we show that the conclusions of Theorems \ref{mainthm} and \ref{mainthm2'} (or Theorem \ref{mainthm2}) fail for two specific classes of binary linear codes known as the Simplex codes $\mathscr{S}_m$ and first-order Reed-Muller Codes $\mathrm{RM}(1,m)$.
\subsection{Definitions of $\mathscr{S}_m$ and $\mathrm{RM}(1,m)$}\label{DefS}
$\mathscr{S}_m$ consists of all linear feedback shift register sequences generated by a single degree $m$ binary primitive polynomial $f$ (these sequences happen to be cyclic shifts of one another---they are called pseudo-noise (PN) sequences or maximal recursive sequences (m-sequences)), together with the zero vector. That is, $\mathscr{S}_m$ is an $[n=2^m-1,m]$ code (hence $\sharp\mathscr{S}_m=2^m=n+1$), and all nonzero vectors have the same weight $2^{m-1}=\frac{n+1}{2}$. It has good cross-correlation properties, and is commonly used in synchronization and scrambling.

We may write algebraically
$$\mathscr{S}_m=\{(\mathrm{Tr}_1^m(ux))_{x \in \mathbb{F}_{2^m}^\times}: u \in \mathbb{F}_{2^m}\}=\{(\mathrm{Tr}_1^m(u\alpha^i))_{i \in [0\isep 2^m-2]}: u \in \mathbb{F}_{2^m}\}$$
where $\alpha \in \mathbb{F}_{2^m}^\times$ is a root of the binary primitive polynomial $f$.

Its dual code is the (perfect) Hamming code $\mathscr{H}_m$ of parameters $[n,n-m,3]_2$. It can be considered as the binary subfield subcode of the algebraic-geometric code $S(D,G)$ with the divisors $D=\sum_{x \in \mathbb{F}_{2^m}^\times} (x)$ and $G=(\infty)-(0)$. Hence by \cite[Theorem 7.2]{AGC}, we have, for all fixed $w \geq 3$,
$$B_w(\mathscr{S}_m)=A_w(\mathscr{H}_m)=\frac{1}{n+1}\binom{n}{w}+O_w(n^{w/2}) \sim \frac{1}{w!}n^{w-1}.$$
In particular while we have $d^\bot=3$, the dual code distribution weight condition of Theorem \ref{mainthm2'} does not hold. We also remark that the code $\mathscr{H}_m$ essentially shows the tightness of the order bound in Lemma \ref{W} for $d=3$ in general.

$\mathrm{RM}(1,m)$ can be constructed as simply a lengthened Simplex code, that is, we insert an overall parity-check bit to (say the leftmost coordinate of) $\mathscr{S}_m$ (since all codewords in the latter have even weight, this parity-check bit is 0 for all these codewords), and then augment the all-one codeword $\mathbf{1}$ and linearize the code. That is,
\begin{equation}\label{RM}
\mathrm{RM}(1,m)=\{(0,\mathbf{c})+v\mathbf{1}: \mathbf{c} \in \mathscr{S}_m, v \in \mathbb{F}_2\}=\{(\mathrm{Tr}_1^m(ux)+v)_{x \in \mathbb{F}_{2^m}}: u \in \mathbb{F}_{2^m}, v \in \mathbb{F}_2\}
\end{equation}
is an $[n=2^m,m+1]$ code. The code size is hence $2^{m+1}=2n$.

Again all codewords of $\mathrm{RM}(1,m)$, except $\mathbf{0}$ and $\mathbf{1}$, have the same weight $2^{m-1}=\frac{n}{2}$.

Its dual code will then be an extended Hamming code $\mathscr{H}_m'$, that is, the $[n,n-m-1]_2$ code formed by inserting an overall parity-check bit to (say the leftmost coordinate of) $\mathscr{H}_m$. In particular, we can easily see that $d^\bot=4$, all weights of $\mathscr{H}_m'$ are even and
$$B_w(\mathrm{RM}(1,m))=A_w(\mathscr{H}_m')=\begin{cases}
0 &(w=2 \text{ or } w \text{ odd})\\
A_{w-1}(\mathscr{H}_m)+A_w(\mathscr{H}_m) \sim \frac{1}{w!}n^{w-1} &(w \geq 4 \text{ and } w \text{ even})
\end{cases}.$$
Here the dual code weight distribution condition in Theorem \ref{mainthm2'} fails too.

\subsection{Empirical Spectral Distribution}
In this section we prove that the empirical spectral distribution of $\widetilde{\mathcal{G}}_{\tilde{s}}$ constructed from either $\mathscr{S}_m$ or $\mathrm{RM}(1,m)$ does not converge to the MP law as $n \to \infty$ (equivalently $m \to \infty$).

By the eigenvalue interlacing property (\ref{E1}), it suffices to prove the following result.

\begin{theorem}\label{thmRM}
    Let $s(1), s(2), \cdots, s(p)$ be $p$ arbitrary (not necessarily distinct) $\pm 1$ sequences derived from either the Simplex code $\mathscr{S}_m$ of length $n=2^m-1$ or the first-order Reed-Muller code $\mathrm{RM}(1,m)$ of length $n=2^m$, with $y:=p/n \in (0,1)$ fixed. Let $\Phi_s$ be the corresponding $p \times n$ matrix, and $\mathcal{G}_s$ be the Gram matrix of $\frac{1}{\sqrt{n}}\Phi_s$. Then the spectral measure $\mu_{\mathcal{G}_s}$ does not converge to the MP law.
\end{theorem}
\begin{proof}
    Recall that the $(i,j)$-th entry of $\mathcal{G}_s$ is given by
    \begin{align*}
        \mathcal{G}_{ij}&=\frac{1}{n}\langle s(i),s(j)\rangle\\
        &=\frac{1}{n}\sum_{t=1}^n s(i)[t]s(j)[t]\\
        &=\frac{1}{n}\sum_{t=1}^n (-1)^{c_{it}+c_{jt}}\\
        &=\frac{1}{n}(n-2\mathrm{wt}(\mathbf{c}_i+\mathbf{c}_j)),
    \end{align*}
    where $\mathbf{c}_i=(c_{i1},\cdots,c_{in})$ and $\mathbf{c}_j=(c_{j1},\cdots,c_{jn})$ are the codewords corresponding to $s(i)$ and $s(j)$ respectively.

    We first consider the $\mathrm{RM}(1,m)$ case. We have
    $$\mathrm{wt}(\mathbf{c})=\begin{cases}
        0 &(\mathbf{c}=\mathbf{0})\\
        2^m=n &(\mathbf{c}=\mathbf{1})\\
        2^{m-1}=\frac{n}{2} &(\text{otherwise})
    \end{cases}.$$
    Hence
    $$\mathcal{G}_{ij}=\begin{cases}
        1 &(\mathbf{c}_i=\mathbf{c}_j)\\
        -1 &(\mathbf{c}_i=\mathbf{c}_j+\mathbf{1})\\
        0 &(\text{otherwise})
    \end{cases}.$$
    In particular, up to permutation of $s$-indices if necessary (this operation preserves the multi-set of eigenvalues), we may write $\mathcal{G}_s$ in block form as
    $$\mathcal{G}_s=\begin{pmatrix}
        \mathcal{J}_1 & O & O & \cdots & O\\
        O & \mathcal{J}_2 & O & \cdots & O\\
        \vdots & \vdots & \vdots & \ddots & \vdots\\
        O & O & O & \cdots & \mathcal{J}_r
    \end{pmatrix},$$
    where $\mathcal{J}_k$'s are symmetric $p_k \times p_k$ matrices satisfying the following conditions:
    \begin{itemize}
        \item $(\mathcal{J}_k)_{ii}=1$ for $i \in [1\isep p_k]$;
        \item $(\mathcal{J}_k)_{ij} \in \{1,-1\}$ for $1 \leq i < j \leq p_k$;
        \item $(\mathcal{J}_k)_{j\l}=(\mathcal{J}_k)_{ij}(\mathcal{J}_k)_{i\l}$ for $1 \leq i < j < \l \leq p_k$,
    \end{itemize}
    $p_1 \geq p_2 \geq \cdots \geq p_r$'s are positive integers with $\sum_{k=1}^r p_k=p$, and $O$'s are zero matrices in compatible dimensions.

    It can be easily checked from the above properties that $\mathcal{J}_k$ has rank one and its only nonzero eigenvalue is $p_k$ (with the first column of $\mathcal{J}_k$ as the corresponding eigenvector). Thereby the eigenvalues of $\mathcal{G}_s$ are given by $p_1,p_2,\cdots,p_r$, and $p-r$ zeros.

Now the spectral measure $\mu_{\mathcal{G}_s}$ has support inside the set of non-negative integers, regardless of $m$ and $s$. In particular $\mu_{\mathcal{G}_s}((0,1))=0$. On the other hand, the lower bound of the support of the (continuous) MP law $\varrho_{\mathrm{MP}_y}$ is $(1-\sqrt{y})^2 \in (0,1)$, since $y \in (0,1)$. Hence $\varrho_{\mathrm{MP}_y}((0,1)) > 0$. Therefore $\mu_{\mathcal{G}_s}$ cannot converge to $\varrho_{\mathrm{MP}_y}$ as $m \to \infty$.

As for the Simplex codes $\mathscr{S}_m$ of length $n=2^m-1$, we have
$$\mathrm{wt}(\mathbf{c})=\begin{cases}
    0 & (\mathbf{c}=\mathbf{0})\\
    2^{m-1}=\frac{n+1}{2} &(\text{otherwise})
\end{cases}.$$
That is,
$$\mathcal{G}_{ij}=\begin{cases}
    1 &(\mathbf{c}_i=\mathbf{c}_j)\\
    -\frac{1}{n} &(\text{otherwise})
\end{cases}.$$
Again, up to permutation of $s$-indices if necessary, we may write $\mathcal{G}_s$ in block form as
$$\mathcal{G}_s=\begin{pmatrix}
    J_{p_1 \times p_1} & -\frac{1}{n}J & -\frac{1}{n}J & \cdots & -\frac{1}{n}J \\
    -\frac{1}{n}J & J_{p_2 \times p_2} & -\frac{1}{n}J & \cdots & -\frac{1}{n}J\\
    \vdots & \vdots & \vdots & \ddots & \vdots\\
    -\frac{1}{n}J & -\frac{1}{n}J & -\frac{1}{n}J & \cdots & J_{p_r \times p_r}
\end{pmatrix},$$
where $p_1 \geq p_2 \geq \cdots \geq p_r$'s are positive integers satisfying $\sum_{k=1}^r p_k=p$, and $J$'s are all-one matrices of compatible dimensions.

Here we may further rewrite $\mathcal{G}_s$ as
$$\mathcal{G}_s=\left(1+\frac{1}{n}\right)\mathcal{G}'-\frac{1}{n}J_{p \times p},$$
where
$$\mathcal{G}'=\begin{pmatrix}
    J_{p_1 \times p_1} & O & O & \cdots & O\\
    O & J_{p_2 \times p_2} & O & \cdots & O\\
    \vdots & \vdots & \vdots & \ddots & \vdots\\
    O & O & O & \cdots & J_{p_r \times p_r}
\end{pmatrix}.$$

Now it is clear that the eigenvalues of $\mathcal{G}'$ are just $p_1,p_2,\cdots,p_r$, and $p-r$ zeros. In particular, the spectral measure $\mu_{(1+\frac{1}{n})\mathcal{G}'}$ of $(1+\frac{1}{n})\mathcal{G}'$ has support in the set $(1+\frac{1}{n})\mathbb{N}_0$. Therefore $\mu_{(1+\frac{1}{n})\mathcal{G}'}((0,1))=0$. Now $\mathcal{G}_s$ is a negative rank-one perturbation of $(1+\frac{1}{n})\mathcal{G}'$. Thus Weyl's interlacing property (see \cite[Corollary 4.3.3]{Horn})
$$\lambda_1' \geq \lambda_1 \geq \lambda_2' \geq \lambda_2 \geq \cdots \geq \lambda_p' \geq \lambda_p$$
holds where $\lambda_i'$'s (which equals $(1+\frac{1}{n})p_i$ when $i \leq r$ and is zero otherwise in this case) are the eigenvalues of $(1+\frac{1}{n})\mathcal{G}'$. This implies
$$0 \leq \mu_{\mathcal{G}_s}((0,1))=|\mu_{\mathcal{G}_s}((0,1))-\mu_{(1+\frac{1}{n})\mathcal{G}'}((0,1))| \leq \frac{1}{p},$$
that is,$$\lim_{m \to \infty}\mu_{\mathcal{G}_s}((0,1))=0 < \varrho_{\mathrm{MP}_y}((0,1)).$$
Hence $\mu_{\mathcal{G}_s}$ cannot converge to $\varrho_{\mathrm{MP}_y}$ either.

This completes the proof of Theorem \ref{thmRM}.
\end{proof}
\subsection{The Largest Eigenvalue $\tilde{\lambda}_1$}
In this section we derive the limiting center and fluctuation of the largest eigenvalue $\tilde{\lambda}_1$ of $\widetilde{\mathcal{G}}_{\tilde{s}}$ for $\mathscr{S}_m$ and $\mathrm{RM}(1,m)$.

\begin{lemma}\label{lemUS}
    Let $\mathcal{C}$ be a binary linear code of length $n$. Let $\mathbf{u}=(u_1,u_2,\cdots,u_p)^T$, where $u_i=\frac{2}{n}\mathrm{wt}(\mathbf{c}_i)$ and $\mathbf{c}_i$ is the codeword corresponding to $\tilde{s}(i)$ for $1 \leq i \leq p$. Assume $y_n=p/n \to y \in (0,1)$. Then
    \begin{enumerate}
        \item If $\mathcal{C}=\mathscr{S}_m$ with $n=2^m-1$, then $\mathbb{V}(\|\mathbf{u}\|^2,\widetilde{\Omega}_p)=y_n+O(n^{-1})$ and
        $$\|\mathbf{u}\|^2-(p+y_n) \xrightarrow{d} y-\mathcal{P}_y,$$
   where $\mathcal{P}_y \sim \mathrm{Poisson}(y)$.
   \item If $\mathcal{C}=\mathrm{RM}(1,m)$ with $n=2^m$, then $\mathbb{V}(\|\mathbf{u}\|^2,\widetilde{\Omega}_p)=5y_n+O(n^{-1})$ and
      $$\|\mathbf{u}\|^2-(p+y_n) \xrightarrow{d} -y+3\mathcal{P}_{1,y/2}-\mathcal{P}_{0,y/2},$$
   where $\mathcal{P}_{1,y/2}$ and $\mathcal{P}_{0,y/2}$ are independent $\mathrm{Poisson}(y/2)$ random variables.
    \end{enumerate}
\end{lemma}
Lemma \ref{lemUS} says that for $\mathcal{C}=\mathscr{S}_m$ or $\mathrm{RM}(1,m)$, $\|\mathbf{u}\|^2$ is not only asymptotically non-Gaussian (but rather discrete), but also has a distinct limiting variance from those codes $\mathcal{C}$ satisfying the conditions of Theorem \ref{mainthm2'} (which include the truly random code $\mathbb{F}_2^n$).

Since $d^\bot \geq 3$ for these two codes, Lemma \ref{lemLUS} still holds. Therefore, by mimicking the proof of Theorem \ref{mainthm2'} with $F$ replaced by the CDF of the corresponding (discrete) limiting distribution stated in Lemma \ref{lemUS} and only considering those $x$'s in which $F$ is continuous (for convergence in distribution checking these points are sufficient---in fact for those $x$ in which $F$ is discontinuous at $x$, that is, the single points in which probability is positive, the equality at (\ref{CDF}) fails to hold), we easily arrive at the following conclusion.
\begin{theorem}\label{thmLS}
\begin{enumerate}
   \item Let $\mathcal{C}=\mathscr{S}_m$ be the binary Simplex code of length $n=2^m-1$. Let $\widetilde{\Phi}_{\tilde{s}}$ be the $p \times n$ 0-1 random matrix as constructed above, with $y_n:=p/n \to y \in (0,1)$. Let $\widetilde{\mathcal{G}}_{\tilde{s}}$ be the Gram matrix of $\frac{2}{\sqrt{n}}\widetilde{\Phi}_{\tilde{s}}$, and $\tilde{\lambda}_1$ denotes its largest eigenvalue. As $n \to \infty$ (equivalently $m \to \infty$), we have
   $$\tilde{\lambda}_1-(p+1+y_n) \xrightarrow{d} y-\mathcal{P}_y,$$
   where $\mathcal{P}_y \sim \mathrm{Poisson}(y)$.

   \item Let $\mathcal{C}=\mathrm{RM}(1,m)$ be the binary first-order Reed-Muller code of length $n=2^m$. Let $\widetilde{\Phi}_{\tilde{s}}$ be the $p \times n$ 0-1 random matrix as constructed above, with $y_n:=p/n \to y \in (0,1)$. Let $\widetilde{\mathcal{G}}_{\tilde{s}}$ be the Gram matrix of $\frac{2}{\sqrt{n}}\widetilde{\Phi}_{\tilde{s}}$, and $\tilde{\lambda}_1$ denotes its largest eigenvalue. As $n \to \infty$ (equivalently $m \to \infty$), we have
   $$\tilde{\lambda}_1-(p+1+y_n) \xrightarrow{d} -y+3\mathcal{P}_{1,y/2}-\mathcal{P}_{0,y/2},$$
   where $\mathcal{P}_{1,y/2}$ and $\mathcal{P}_{0,y/2}$ are independent $\mathrm{Poisson}(y/2)$ random variables.
\end{enumerate}
\end{theorem}
Theorem \ref{thmLS} says that the limiting center is the same as the truly random case, but the fluctuation is non-Gaussian (discrete).

The rest of the section is then devoted to the proof of Lemma \ref{lemUS}.
\begin{proof}[Proof of Lemma \ref{lemUS}]
Let $\mathbf{v}=\mathbf{1}_p-\mathbf{u}$ and recall the definition of $Z$ in (\ref{Z0}) and the equality (\ref{U2}).
\begin{enumerate}
    \item We have
    \begin{equation}\label{V1}
    v_i=1-\frac{2}{n}\mathrm{wt}(\mathbf{c}_i)=\begin{cases}
        1 &(\mathbf{c}_i=\mathbf{0})\\
        1-\frac{2}{n}\left(\frac{n+1}{2}\right)=-\frac{1}{n} &(\text{otherwise})
    \end{cases}.
    \end{equation}
    Hence
    \begin{align}
        \|\mathbf{v}\|^2&=\sum_{i=1}^p v_i^2\nonumber\\
        &=\mathcal{B}+\frac{1}{n^2}(p-\mathcal{B})\nonumber\\
        &=\frac{y_n}{n}+\left(1-\frac{1}{n^2}\right)\mathcal{B}\label{V2}
    \end{align}
    and
    \begin{align}
        Z&=-2\sum_{i=1}^p v_i\nonumber\\
        &=-2\left(\mathcal{B}-\frac{1}{n}(p-\mathcal{B})\right)\nonumber\\
        &=2y_n-2\left(1+\frac{1}{n}\right)\mathcal{B},\label{Z2}
    \end{align}
    where
    $$\mathcal{B}:=\sharp\{i \in [1\isep p]: \mathbf{c}_i=\mathbf{0}\}.$$
    (\ref{U2}) yields
\begin{align*}
    \|\mathbf{u}\|^2-(p+y_n)&=(\|\mathbf{v}\|^2-y_n)+Z\\
    &=y_n+\frac{y_n}{n}-\left(1+\frac{1}{n}\right)^2\mathcal{B}.
    \end{align*}
    Since $\mathbf{c}_i$'s are independent, the random variable $\mathcal{B}$ follows a binomial distribution on $p$ objects with probability $\P(\mathbf{c}_i=\mathbf{0},\widetilde{\Omega}_p)=\frac{1}{n+1}$.

    Hence
    \begin{align*}
    \mathbb{V}(\|\mathbf{u}\|^2,\widetilde{\Omega}_p)&=\mathbb{V}\left(y_n+\frac{y_n}{n}-\left(1+\frac{1}{n}\right)^2\mathcal{B},\widetilde{\Omega}_p\right)\\
    &=\left(1+\frac{1}{n}\right)^4\mathbb{V}(\mathcal{B},\widetilde{\Omega}_p)\\
    &=\left(1+\frac{1}{n}\right)^4p\left(\frac{1}{n+1}\right)\left(1-\frac{1}{n+1}\right)\\
    &=y_n+O\left(\frac{1}{n}\right).
        \end{align*}

Moreover, since $\frac{p}{n+1} \to y > 0$ as $n \to \infty$, we have $\mathcal{B} \xrightarrow{d} \mathcal{P}_y$ where $\mathcal{P}_y \sim \mathrm{Poisson}(y)$. Therefore
    $$\|\mathbf{u}\|^2-(p+y_n) \xrightarrow{d} y-\mathcal{P}_y$$
    as desired.
\item We have
    \begin{equation}\label{V3}
    v_i=1-\frac{2}{n}\mathrm{wt}(\mathbf{c}_i)=\begin{cases}
        1 &(\mathbf{c}_i=\mathbf{0})\\
        -1 &(\mathbf{c}_i=\mathbf{1})\\
        0 &(\text{otherwise})
    \end{cases}.
    \end{equation}
    Therefore
    \begin{equation}\label{V4}
    \|\mathbf{v}\|^2=\sum_{i=1}^pv_i^2=\mathcal{B}_0+\mathcal{B}_1
    \end{equation}
    and
    \begin{equation}\label{Z3}
    Z=-2\sum_{i=1}^p v_i=2\mathcal{B}_1-2\mathcal{B}_0,
    \end{equation}
    where
    $$\mathcal{B}_0:=\sharp\{i \in [1\isep p]: \mathbf{c}_i=\mathbf{0}\}$$
    and
    $$\mathcal{B}_1:=\sharp\{i \in [1\isep p]: \mathbf{c}_i=\mathbf{1}\}$$
    respectively.
    
    By (\ref{U2}), we obtain
$$\|\mathbf{u}\|^2-(p+y_n)=(\|\mathbf{v}\|^2-y_n)+Z\\
    =-y_n+3\mathcal{B}_1-\mathcal{B}_0.$$
    Since $\mathbf{c}_i$'s are independent, the random variables $\mathcal{B}_1$ and  $\mathcal{B}_0$ follow binomial distributions on $p$ objects with probability $\P(\mathbf{c}_i=\mathbf{1},\widetilde{\Omega}_p)=\frac{1}{2n}$ and $\P(\mathbf{c}_i=\mathbf{0},\widetilde{\Omega}_p)=\frac{1}{2n}$ respectively.

    Moreover,
    \begin{align*}
        \mathrm{Cov}(\mathcal{B}_1,\mathcal{B}_0,\widetilde{\Omega}_p):&=\E(\mathcal{B}_1\mathcal{B}_0,\widetilde{\Omega}_p)-\E(\mathcal{B}_1,\widetilde{\Omega}_p)\E(\mathcal{B}_0,\widetilde{\Omega}_p)\\
        &=\E\left(\sum_{i,j=1}^p\mathbbm{1}_{\mathbf{c}_i=\mathbf{1}}\mathbbm{1}_{\mathbf{c}_j=\mathbf{0}},\widetilde{\Omega}_p\right)-\left(\frac{p}{2n}\right)^2\\
        &=\sum_{i=1}^p\E(\mathbbm{1}_{\mathbf{c}_i=\mathbf{1}}\mathbbm{1}_{\mathbf{c}_i=\mathbf{0}},\widetilde{\Omega}_p)+\sum_{1 \leq i \neq j \leq p}\E(\mathbbm{1}_{\mathbf{c}_i=\mathbf{1}}\mathbbm{1}_{\mathbf{c}_j=\mathbf{0}},\widetilde{\Omega}_p)-\left(\frac{p}{2n}\right)^2\\
        &=p(p-1)\left(\frac{1}{2n}\right)^2-\left(\frac{p}{2n}\right)^2\\
        &=-\frac{p}{4n^2}.
    \end{align*}
    Thus
    \begin{align*}
        \mathbb{V}(\|\mathbf{u}\|^2,\widetilde{\Omega}_p)&=\mathbb{V}(-y_n+3\mathcal{B}_1-\mathcal{B}_0,\widetilde{\Omega}_p)\\
        &=9\mathbb{V}(\mathcal{B}_1,\widetilde{\Omega}_p)+\mathbb{V}(\mathcal{B}_0,\widetilde{\Omega}_p)-6\mathrm{Cov}(\mathcal{B}_1,\mathcal{B}_0,\widetilde{\Omega}_p)\\
        &=10p\left(\frac{1}{2n}\right)\left(1-\frac{1}{2n}\right)+\frac{3p}{2n^2}\\
        &=5y_n+O\left(\frac{1}{n}\right).
    \end{align*}
    Since $\frac{p}{2n} \to \frac{y}{2} > 0$, we see that $\mathcal{B}_1,\mathcal{B}_0 \xrightarrow{d} \mathrm{Poisson}(\frac{y}{2})$. 
    
    It remains to show that $\mathcal{B}_1$ and $\mathcal{B}_0$ are asymptotically independent. Indeed, the joint probability generating function of $(\mathcal{B}_1,\mathcal{B}_0)$ is given by
   $$G_{\mathcal{B}_1,\mathcal{B}_0}(t_1,t_0)=\left(\frac{t_1}{2n}+\frac{t_0}{2n}+\left(1-\frac{1}{n}\right)\right)^p=\left(1+\frac{t_1+t_0-2}{2n}\right)^p.$$
   Hence
   $$\lim_{m \to \infty} G_{\mathcal{B}_1,\mathcal{B}_0}(t_1,t_0)=e^\frac{(t_1+t_0-2)y}{2}=e^\frac{(t_1-1)y}{2}e^\frac{(t_0-1)y}{2}=\left(\lim_{m \to \infty}G_{\mathcal{B}_1}(t_1)\right)\left(\lim_{m \to \infty}G_{\mathcal{B}_0}(t_0)\right).$$
   This completes the proof of Lemma \ref{lemUS}.
\end{enumerate}
\end{proof}
\begin{remark}
Recall the comparison with \cite{Bryc} in Remark \ref{Rk2}. If (\ref{sigma}) with $\ep^{(n)}=o_P(1)$ (equivalently (\ref{L4}) with $\ep^{'(n)}=o_P(1)$) were true for all $\mathcal{C}$ with $d^\bot \geq 3$ (or equivalently, pairwise independence within individual rows), then we would also have $\|\mathbf{v}\|^2=y_n+o_P(1)$ by (\ref{L3}).

However, by (\ref{V1}), (\ref{V2}), (\ref{V3}) and (\ref{V4}), for $\mathcal{C} \in \{\mathscr{S}_m, \mathrm{RM}(1,m)\}$, we have
\begin{align*}
\E(v_i^4,\widetilde{\Omega}_p)&=\frac{1}{n}+O\left(\frac{1}{n^2}\right),\\
\mathbb{V}(\|\mathbf{v}\|^2,\widetilde{\Omega}_p)&=y_n+O\left(\frac{1}{n}\right) \text{ and}\\
\|\mathbf{v}\|^2 &\xrightarrow{d} \mathcal{P}_y,
\end{align*}
which not only disqualifies the claim in the former paragraph, but also shows that the condition $B_4(\mathcal{C})=o(n^3)$ cannot be dropped in the second half of Lemma \ref{lemV}.

Hence even pairwise independence within row entries is insufficient to guarantee that (\ref{sigma}) holds with $\ep^{(n)}=o_P(1)$. This also explains why we do not directly adapt the ideas of proof in \cite{Bryc,Silverstein}.

For the sake of completeness in comparison with the proof of Lemma \ref{lemU}, we also derive from (\ref{V2}), (\ref{Z2}), (\ref{V4}) and (\ref{Z3}) that, for $\mathcal{C}=\mathscr{S}_m$,
$$Z \xrightarrow{d} 2y-2\mathcal{P}_y$$
and
$$\mathrm{Cov}(\|\mathbf{v}\|^2,Z,\widetilde{\Omega}_p)=-2\left(1-\frac{1}{n^2}\right)\left(1+\frac{1}{n}\right)\mathbb{V}(\mathcal{B},\widetilde{\Omega}_p)=-2y_n+O\left(\frac{1}{n}\right)$$
respectively, and for $\mathcal{C}=\mathrm{RM}(1,m)$, we have
$$Z \xrightarrow{d} 2\mathcal{P}_{1,y/2}-2\mathcal{P}_{0,y/2}$$
and
$$\mathrm{Cov}(\|\mathbf{v}\|^2,Z,\widetilde{\Omega}_p)=2\mathrm{Cov}(\mathcal{B}_0+\mathcal{B}_1,\mathcal{B}_1-\mathcal{B}_0,\widetilde{\Omega}_p)=2(\mathbb{V}(\mathcal{B}_1,\widetilde{\Omega}_p)-\mathbb{V}(\mathcal{B}_0,\widetilde{\Omega}_p))=0$$
respectively. In particular absolute equality holds on the Cauchy-Schwarz inequality in (\ref{Cov}) for $\mathscr{S}_m$, since in that case both $\|\mathbf{v}\|^2$ and $Z$ are linear with respect to the same single random variable $\mathcal{B}$.
\end{remark}
\section{Numerical Simulations}\label{ns}
In this section we present numerical simulations through \textbf{MATLAB}.

\subsection{Binary Gold Codes $\mathscr{G}_m$}
The first set of experiments is based on binary Gold codes $\mathscr{G}_m$.

A Gold sequence is the binary addition (or XOR) of two PN sequences (m-sequences) defined by two distinct primitive polynomials $f_1(x)$ and $f_2(x)$ of the same degree $m$ \cite{Gold}. Let $\alpha \in \mathbb{F}_{2^m}^\times$ be a root of $f_1$, and $\delta$ be a positive integer such that $\alpha^\delta$ is a root of $f_2$. In \textbf{MATLAB}, a ``preferred pair'' of primitive polynomials $f_1$ and $f_2$ is defined whenever $4 \nmid m$, if $\delta=2^h+1$ or $2^{2h}-2^h+1$, for some positive integer $h$ with $\gcd(h,m)=1$ or 2 according to whether $m$ is odd or even. In fact here the number $2^h+1$ is called a \emph{Gold exponent} and the corresponding monomial $x^{2^h+1}$ is called a \emph{Gold function} (the another exponent $2^{2h}-2^h+1$ is called a \emph{Kasami exponent} and the corresponding monomial $x^{2^{2h}-2^h+1}$ is a \emph{Kasami function}). In the following we focus on the case that uses the Gold functions, which is also the construction made by Gold \cite{Gold} when he introduced this class of codes.

If $\mathbf{u}$ and $\mathbf{v}$ are PN sequences defined by $f_1$ and $f_2$ respectively, then the set of Gold sequences $\mathscr{G}_m(\mathbf{u},\mathbf{v})$ is defined as
$$\mathscr{G}_m(\mathbf{u},\mathbf{v}):=\{T^s\mathbf{u},T^s\mathbf{v},T^s\mathbf{u}\oplus T^t\mathbf{v}: s,t \in [0\isep 2^m-2]\}$$
where $T$ and $\oplus$ denote cyclic shift and binary addition (XOR) operations respectively. In particular $\mathscr{G}_m(\mathbf{u},\mathbf{v})$ consists of $2^{2m}-1$ binary sequences, with $2^m-1$ being PN sequences derived from $f_1$, $2^m-1$ being PN sequences derived from $f_2$, and $(2^m-1)^2$ being non-trivial Gold sequences. All these sequences possess nice cross-correlation properties. In particular, the world-renowned global positioning system (GPS) and the code division multiple access (CDMA) in mobile phones use Gold sequences with $m=10$ and 15 (that is, length $n=1023$ and 32767) respectively.

Algebraically, we may write these Gold sequences as
$$\mathbf{c}_{u,v}^h:=(\mathrm{Tr}_1^m(ux+vx^{2^h+1}))_{x \in \mathbb{F}_{2^m}^\times}$$
for $u,v \in \mathbb{F}_{2^m}$, not both zero.

The set $\mathscr{G}_m:=\mathscr{G}_m(\mathbf{u},\mathbf{v}) \cup \{\mathbf{0}\}=\{\mathbf{c}_{u,v}^h: u,v \in \mathbb{F}_{2^m}\}$ forms a binary linear code with parameters $[n=2^m-1,2m]$, which is what we call a binary Gold code.

A remarkable thing to note here is that, when $\gcd(h,m)=1$ (in this case the Gold function $x^{2^h+1}$ is so-called almost perfect nonlinear [APN], a property with important applications in cryptography), the dual distance of $\mathscr{G}_m$ is exactly 5 (see \cite{Gold} and \cite[Lemma 10 and Eq. (25)]{Kasami}; in particular, if we take $h=1$, then $\mathscr{G}_m$ is just the dual code of a primitive narrow-sense double-error-correcting BCH code). However, when $\gcd(h,m)=2$, it happens that the dual distance of $\mathscr{G}_m$ is 3 instead. Nevertheless, we have $B_4(\mathscr{G}_m) \asymp n^2$ by \cite[Lemma 10 and Eq. (25)]{Kasami}.
     \begin{table}[ht!]
    \centering
    \begin{tabular}{|c||c|c|c|}
    \hline
    $m$ & $f_1$ & $f_2$ & $h$ \\
    \hline
    \hline
    5 & $x^5+x^2+1$ & $x^5+x^4+x^3+x^2+1$ & 1\\
    \hline
    6 & $x^6+x+1$ & $x^6+x^5+x^2+x+1$ & 2\\
    \hline
    7 & $x^7+x^3+1$ & $x^7+x^3+x^2+x+1$ & 1\\
    \hline
    10 & $x^{10}+x^3+1$ & $x^{10}+x^8+x^3+x^2+1$ & 2\\
    \hline
    11 & $x^{11}+x^2+1$ & $x^{11}+x^8+x^5+x^2+1$ & 1\\
    \hline
    13 & $x^{13}+x^4+x^3+x+1$ & $x^{13}+x^{10}+x^9+x^7+x^5+x^4+1$ & 1\\
    \hline
    14 & $x^{14}+x^{10}+x^6+x+1$ & $x^{14}+x^{10}+x^9+x^7+x^6+x^4+x^3+x+1$ & 2\\
    \hline
    \end{tabular}
    \caption{The preferred pair of binary primitive polynomials in \textbf{MATLAB} and the corresponding $h$ for the Gold exponent $2^h+1$ used in the construction of Gold sequences and matrices $\widetilde{\Phi}_{\tilde{s}}$ in our numerical simulations}
    \label{t1}
\end{table}

In our numerical simulation on the ESD, we use binary Gold codes with $m=10,11,13,14$. We use a preferred pair of binary primitive polynomials in \textbf{MATLAB} (see Table \ref{t1} for the list) to generate $p=\lfloor n/2\rfloor=2^{m-1}-1$ Gold sequences (so $y_n \sim 0.5$ in each case), with the first $m$ bits of each involved PN sequence chosen at random uniformly and independently. Then we run one realization of $\widetilde{\mathcal{G}}_{\tilde{s}}$ per experiment. We repeat this for 10 times for each $m$ and compute the Kolmogorov-Smirnov (KS) distance of the ESD CDF with the MP CDF (that is, $\sup_{x \in \mathbb{R}}|\mu_{\widetilde{\mathcal{G}}_{\tilde{s}}}((-\infty,x])-\varrho_{\mathrm{MP}_{y_n}}((-\infty,x])|$). The empirical mean and variance of the KS distance of these 10 realizations are recorded in Table \ref{t2}. We also generate a plot of the ESD and MP CDFs and also a histogram of the ESD together with the MP PDF for one realization of each $m$ to illustrate the MP-like ``shape'' of the ESD in Figure \ref{fig3} (here we remark that, the ESD and histogram plots for the other 9 realizations of each $m$ also resemble similar shapes. This matches the ``high probability'' nature of the convergence in Theorem \ref{mainthm}).

 \begin{table}[ht!]
    \centering
    \begin{tabular}{|c||c|c|c|c|c|c|c|}
    \hline
    $m$ & $n$ & $p$ & mean of KS distance & variance of KS distance \\
    \hline
    \hline
    10 & 1023 & 511 & 0.0047304 & $5.92941 \times 10^{-7}$ \\
    \hline
    11 & 2047 & 1023 & 0.0025705 & $1.31155 \times 10^{-7}$\\
    \hline
    13 & 8191 & 4095 & 0.00091 & $1.29422 \times 10^{-8}$ \\
    \hline
    14 & 16383 & 8191 & 0.0005774 & $2.41693 \times 10^{-9}$ \\
    \hline
    \end{tabular}
    \caption{Empirical statistics of KS distances between ESD CDF of $\widetilde{\Phi}_{\tilde{s}}$ and the MP CDF in 10 realizations of numerical simulations based on Gold sequences, $n=2^m-1, p=2^{m-1}-1$}
    \label{t2}
\end{table}
As seen from the table, as $m$ increases, the mean and variance of the KS distance indeed are smaller, with the variance having a much smaller magnitude as well as faster decay rate. Note that here we include the cases that $d^\bot=3$ as well ($m=10,14$). These cases are not covered by our Theorem \ref{mainthm}. If we only consider convergence of ESD to MP law (without the rate), the only best known positive situation with $d^\bot \geq 3$ is, to control the error term $E_\l$ in \cite[Eq. (3)]{OQBT} to $o_\l(1)$. This is the case provided that $B_4(\mathcal{C})=o(n^2)$, but we only have $B_4(\mathcal{C}) \asymp n^2$ as explained earlier. The simulation result suggests that there may be a tighter upper bound for $E_\l$. See Section \ref{op} for details.

Moreover, we note that while our theorem gives the order $O_\prec(n^{-1/4})$, the mean in Table \ref{t2} actually suggests an even faster decay, around of order $O(n^{-3/4})$. This also raises an open problem for the optimal convergence rate under several code conditions. This is discussed in Section \ref{op} for more details.

As for the largest eigenvalue $\tilde{\lambda}_1$, we use Gold sequences with $m=5,6,7$ (the preferred pair of binary primitive polynomials as well as the $h$ in the Gold exponent used in \textbf{MATLAB} is shown in Table \ref{t1}) to run 500 realizations of $\widetilde{\Phi}_{\tilde{s}}$ in a single experiment. For convenience we normalize $\tilde{\lambda}_1$ by the theoretical mean and variance $p+1+y_n$ and $4y_n$ respectively, that is, we consider the distribution of
\begin{equation}\label{NLE}
    \tilde{\Lambda}_1:=\frac{\tilde{\lambda}_1-(p+1+y_n)}{2\sqrt{y_n}}
\end{equation}
and compare it with the standard normal (Gaussian) distribution $\mathcal{N}(0,1)$.

We plot histograms of the empirical distribution of $\tilde{\Lambda}_1$, together with the density of the standard normal distribution, to illustrate the agreement of the two distributions in Figure \ref{fig4}. For more precise demonstration, we also provide the mean, variance and KS distance for the data of the 500 $\tilde{\Lambda}_1$'s in Table \ref{t3}, as well as quantile-quantile (Q--Q) plots between the empirical distribution of $\tilde{\Lambda}_1$ and $\mathcal{N}(0,1)$ in Figure \ref{fig4'}. As seen from the table, the empirical means and variances are close to 0 and 1 respectively. The KS distances are around 0.02--0.04, which are within the 95\% confidence threshold $\frac{K_{0.05}}{\sqrt{500}} \sim 0.060732$. The Q--Q plots also show a good agreement to the reference straight line $y=x$. These are all consistent with Theorem \ref{mainthm2} for $m=5,7$ and Theorem \ref{mainthm2'} for all three $m$'s.

\begin{table}[ht!]
    \centering
    \begin{tabular}{|c||c|c|c|c|c|}
    \hline
    $m$ & $n$ & $p$ & mean & variance & KS distance with $\mathcal{N}(0,1)$\\
    \hline
    \hline
    5 & 31 & 15 & 0.0532 & 1.0301 & 0.040237 \\
    \hline
    6 & 63 & 31 & 0.0174 & 0.9947 & 0.020644\\
    \hline
    7 & 127 & 63 & -0.0135 & 1.0302 & 0.020299 \\
    \hline
    \end{tabular}
    \caption{Empirical mean, variance and KS distance with $\mathcal{N}(0,1)$ of $\tilde{\Lambda}_1$ for 500 realizations of numerical simulations based on Gold sequences, $n=2^m-1, p=2^{m-1}-1$}
    \label{t3}
\end{table}

\begin{figure}[ht!]
\includegraphics[angle=0,width=0.5 \textwidth,height=0.2 \textheight]{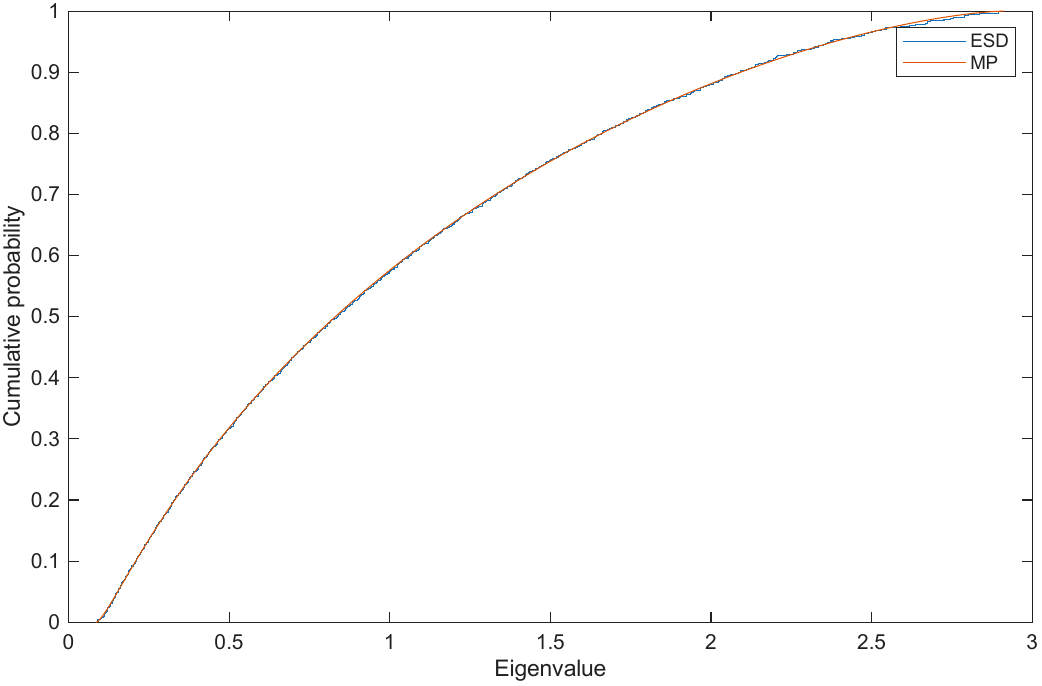}
\includegraphics[angle=0,width=0.5 \textwidth,height=0.2 \textheight]{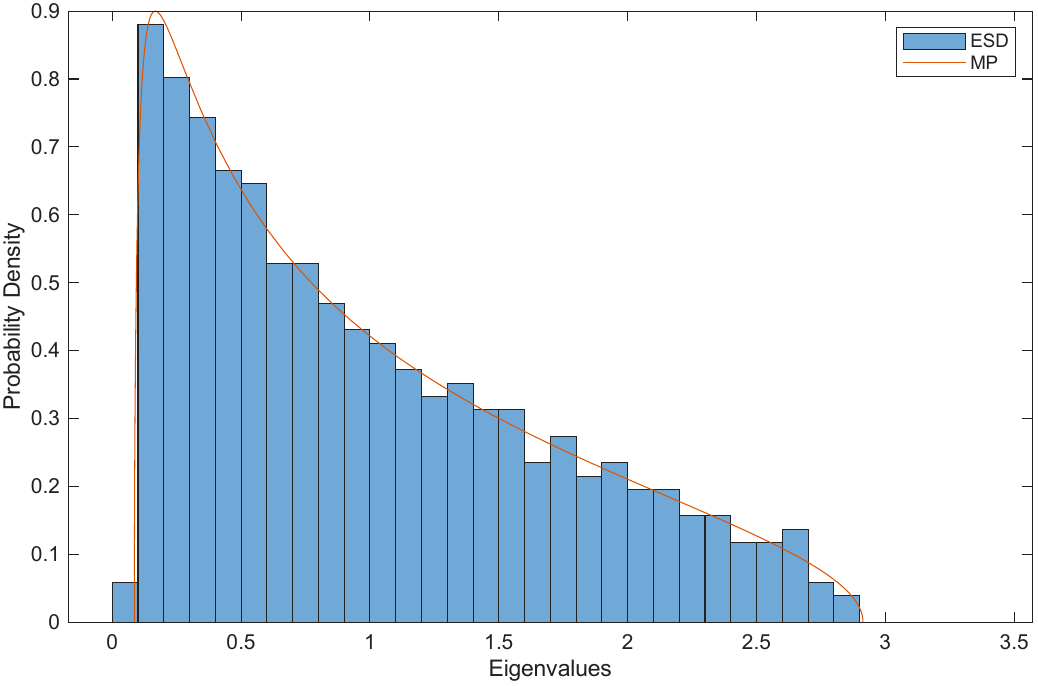}
\newline
\includegraphics[angle=0,width=0.5 \textwidth,height=0.2 \textheight]{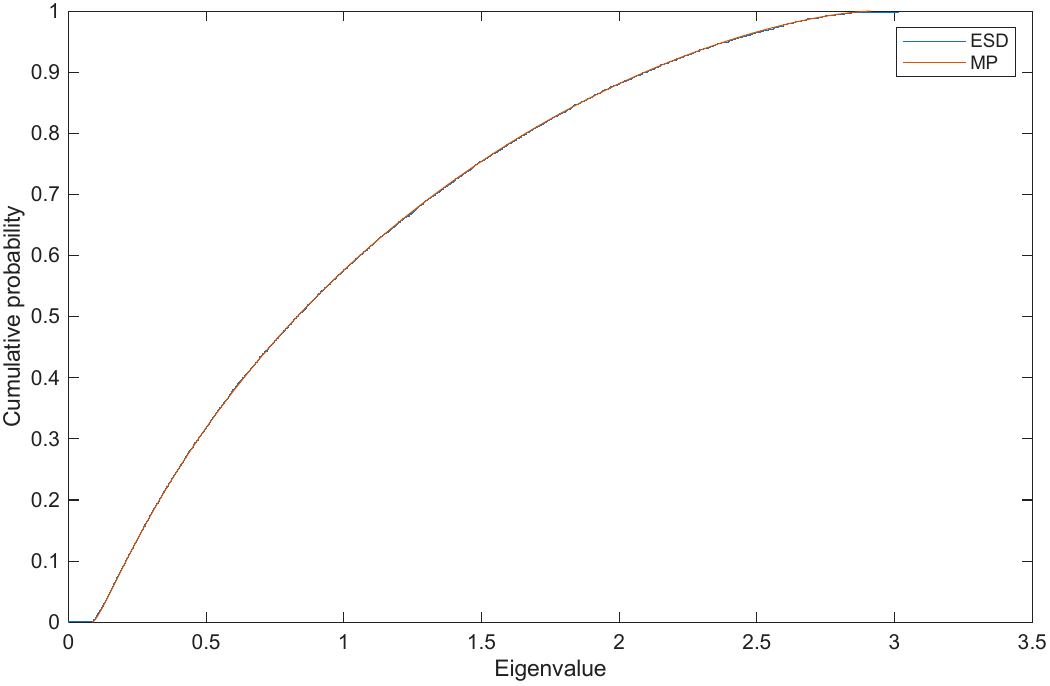}
\includegraphics[angle=0,width=0.5 \textwidth,height=0.2 \textheight]{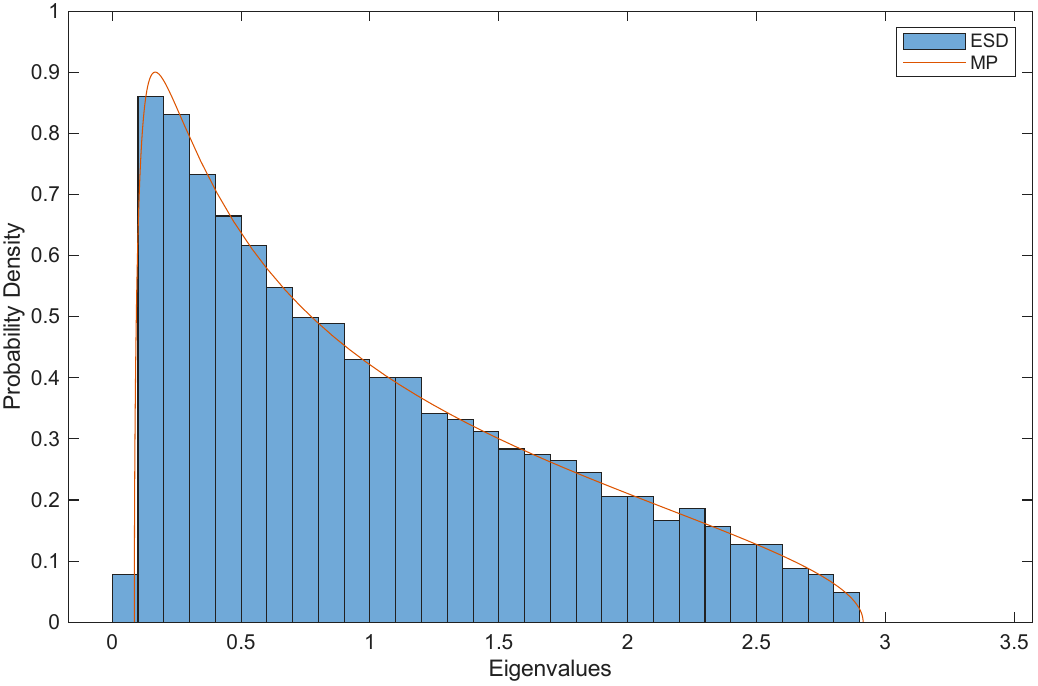}
\newline
\includegraphics[angle=0,width=0.5 \textwidth,height=0.2 \textheight]{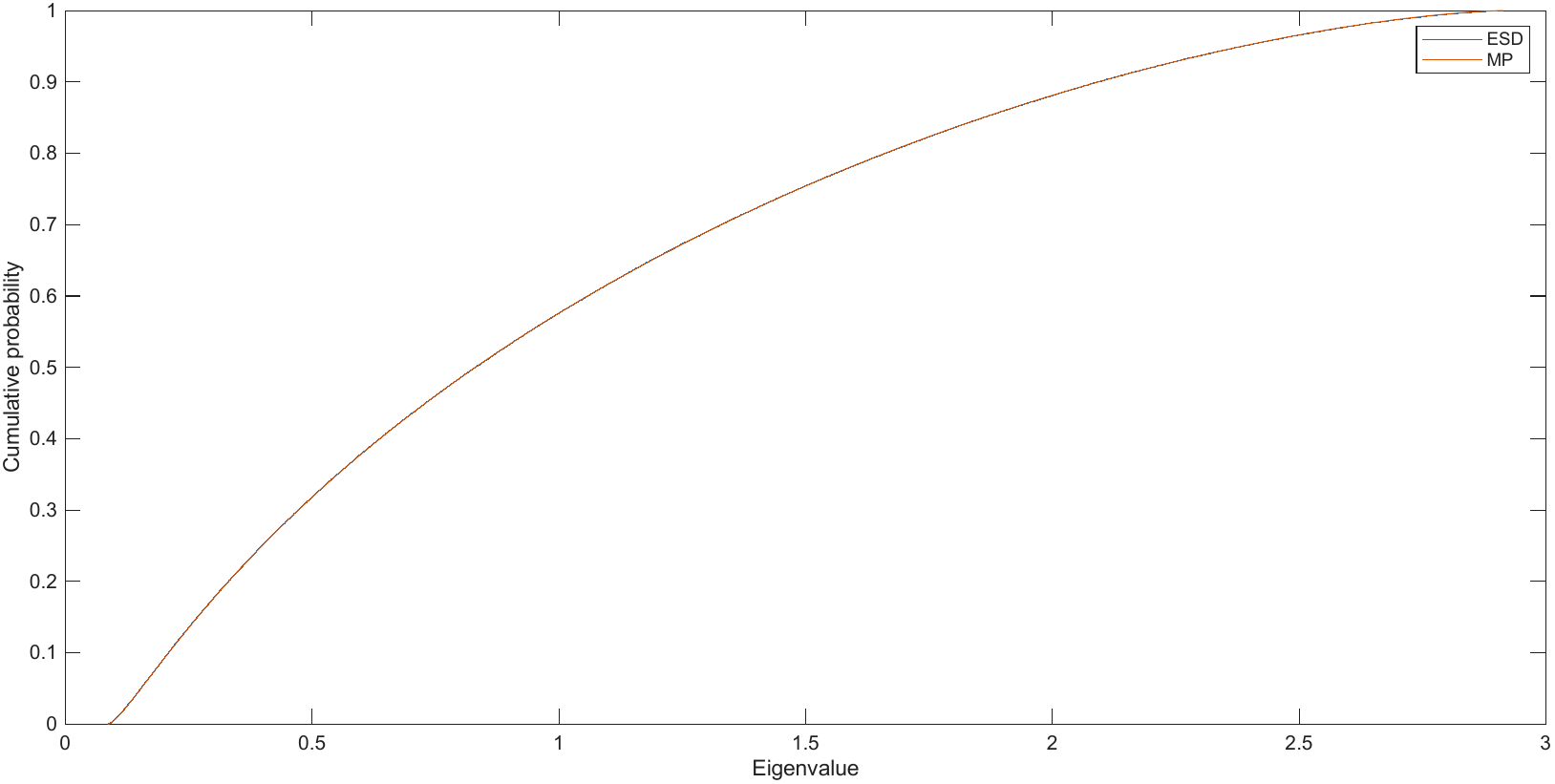}
\includegraphics[angle=0,width=0.5 \textwidth,height=0.2 \textheight]{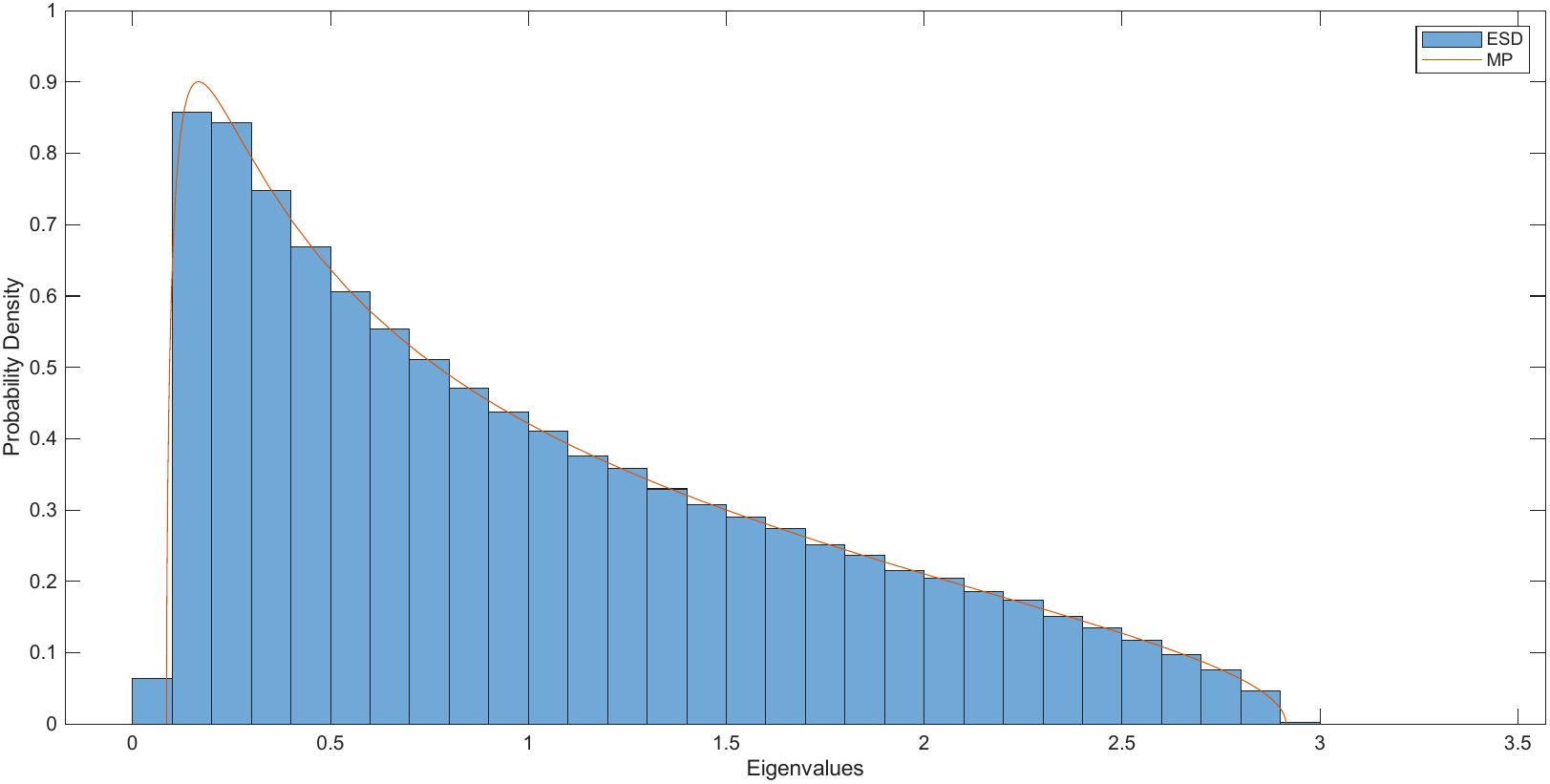}
\newline
\includegraphics[angle=0,width=0.5 \textwidth,height=0.2 \textheight]{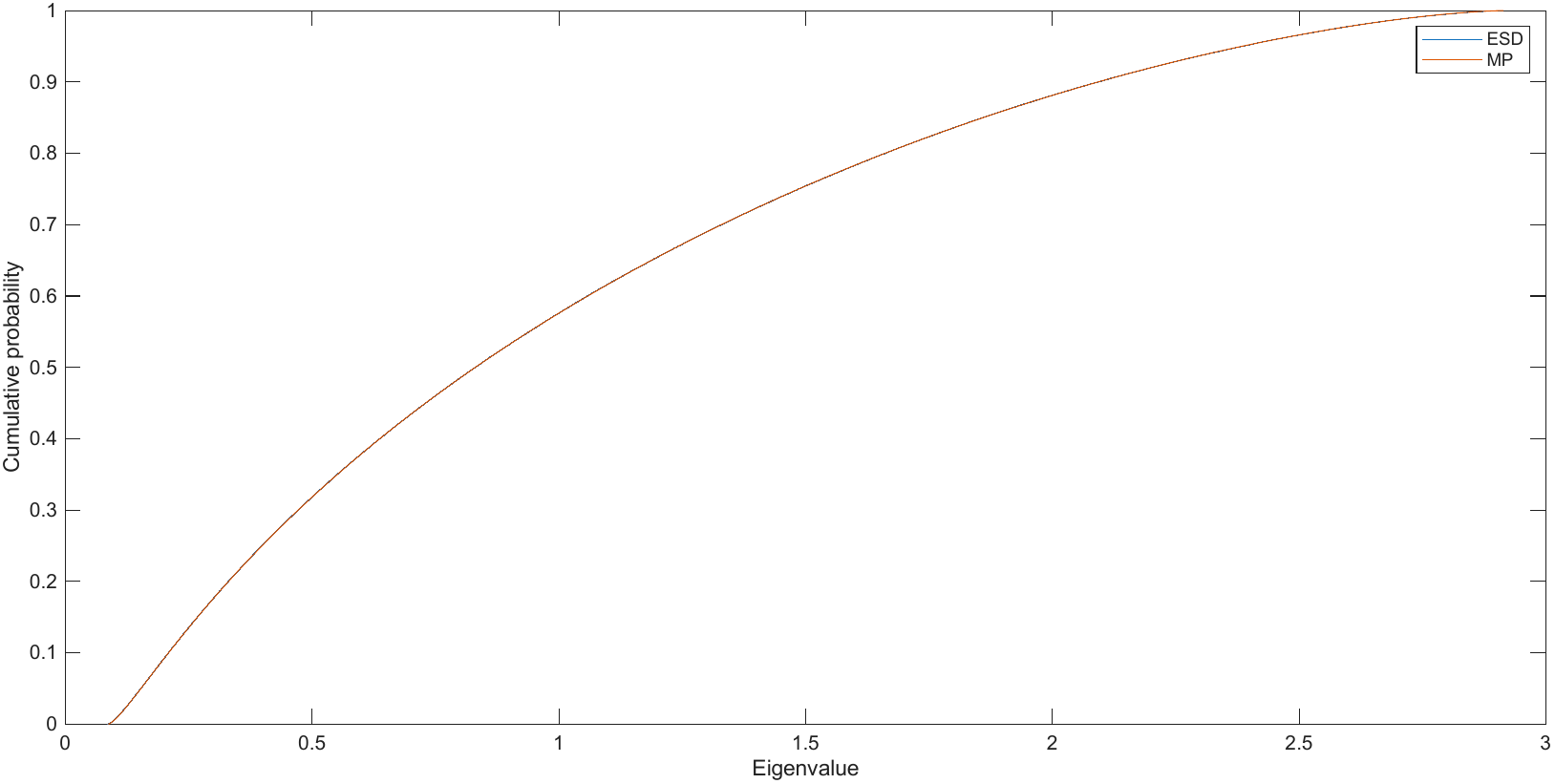}
\includegraphics[angle=0,width=0.5 \textwidth,height=0.2 \textheight]{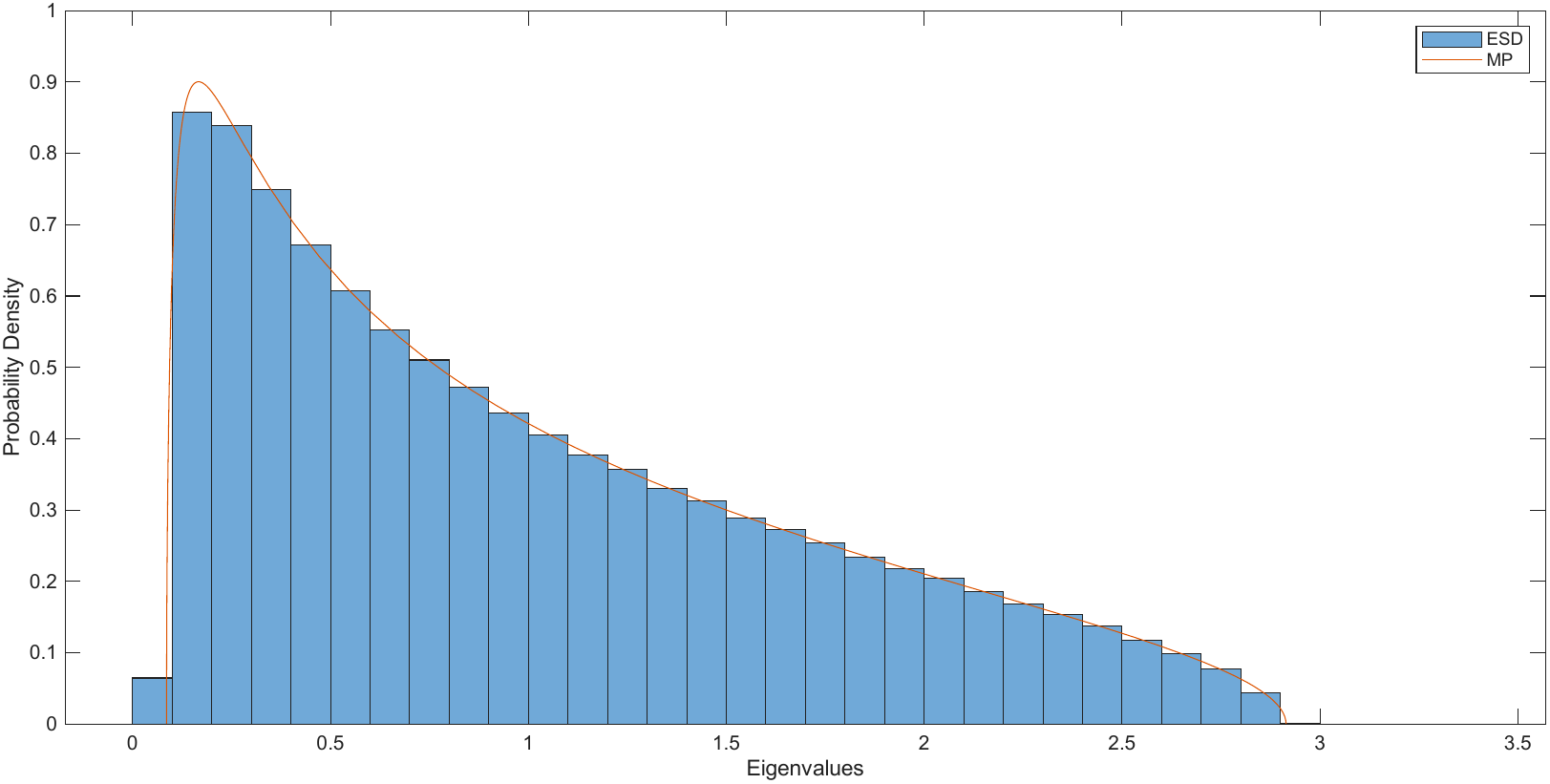}
\caption{ESD (CDF and PDF) of $\widetilde{\mathcal{G}}_{\tilde{s}}$ based on $[1023,20], [2047,22],[8191,26]$ and $[16383,28]$ binary Gold codes, with $(p,n,m)=(511,1023,10),(1023,2047,11),(4095,8191,13)$ and $(8191,16383,14)$.}\label{fig3}
\end{figure}
\begin{figure}[ht!]
\includegraphics[angle=0,width=0.32 \textwidth,height=0.2 \textheight]{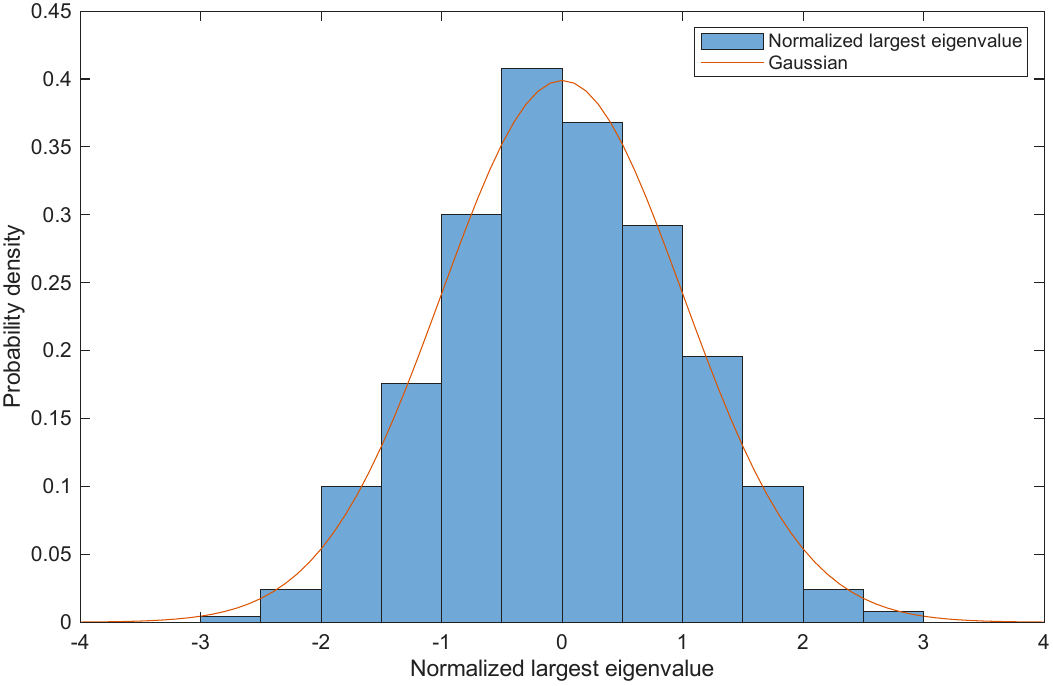}
\includegraphics[angle=0,width=0.32 \textwidth,height=0.2 \textheight]{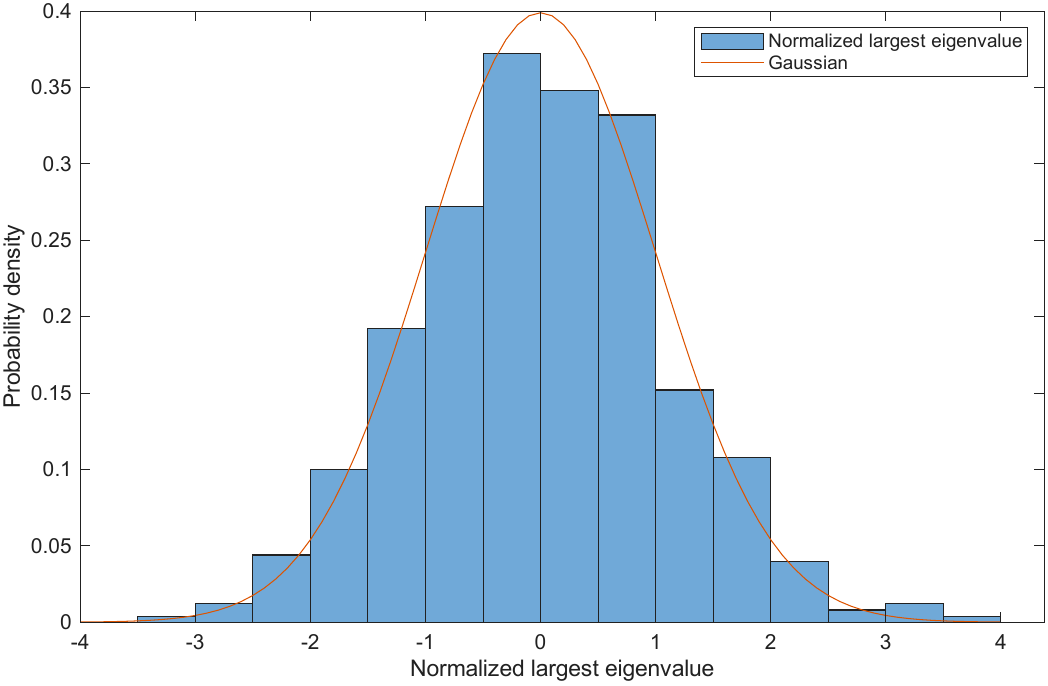}
\includegraphics[angle=0,width=0.32 \textwidth,height=0.2 \textheight]{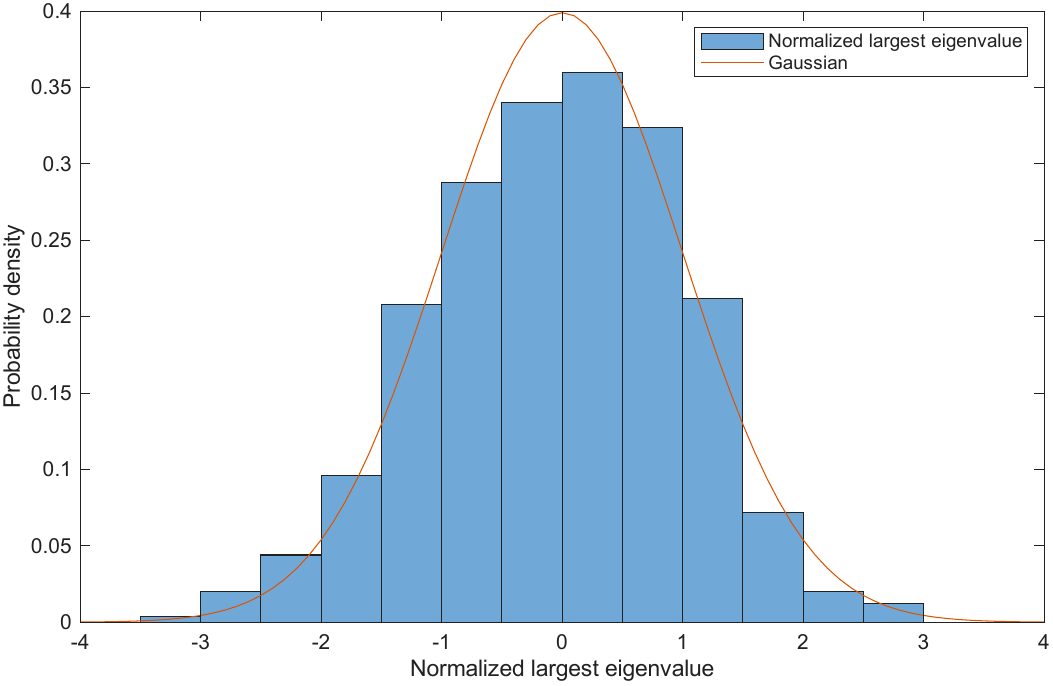}
\caption{PDF of normalized largest eigenvalue of 500 realizations of $\widetilde{\mathcal{G}}_{\tilde{s}}$ based on $[31,10],[63,12]$ and $[127,14]$ binary Gold codes, with $(p,n,m)=(15,31,5), (31,63,6)$ and $(63,127,7)$.}\label{fig4}
\end{figure}
\begin{figure}[ht!]
\includegraphics[angle=0,width=0.32 \textwidth,height=0.2 \textheight]{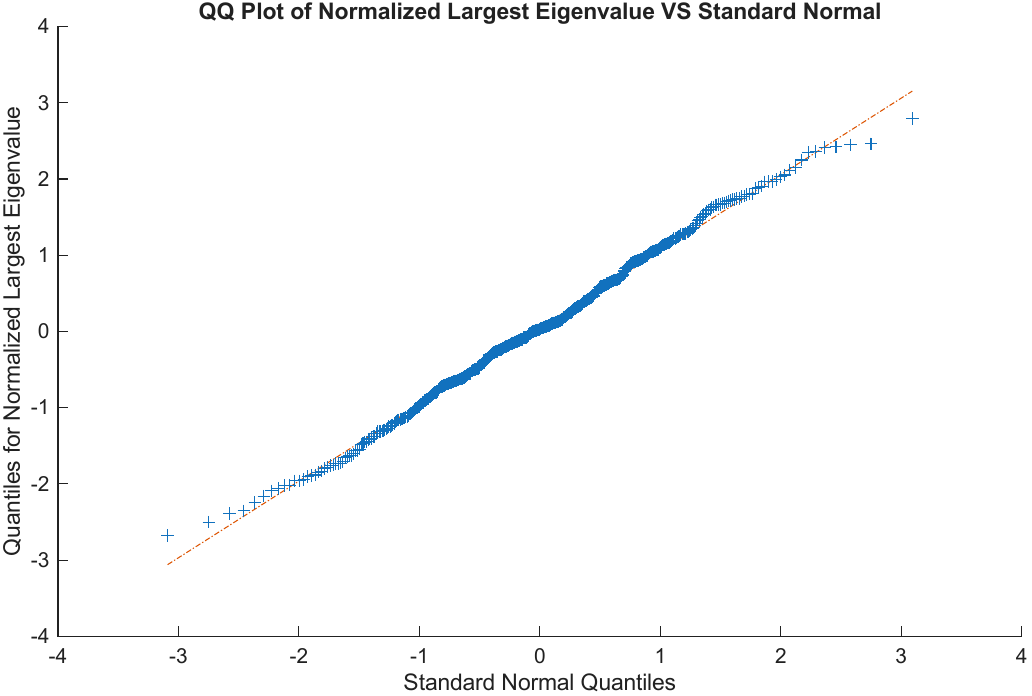}
\includegraphics[angle=0,width=0.32 \textwidth,height=0.2 \textheight]{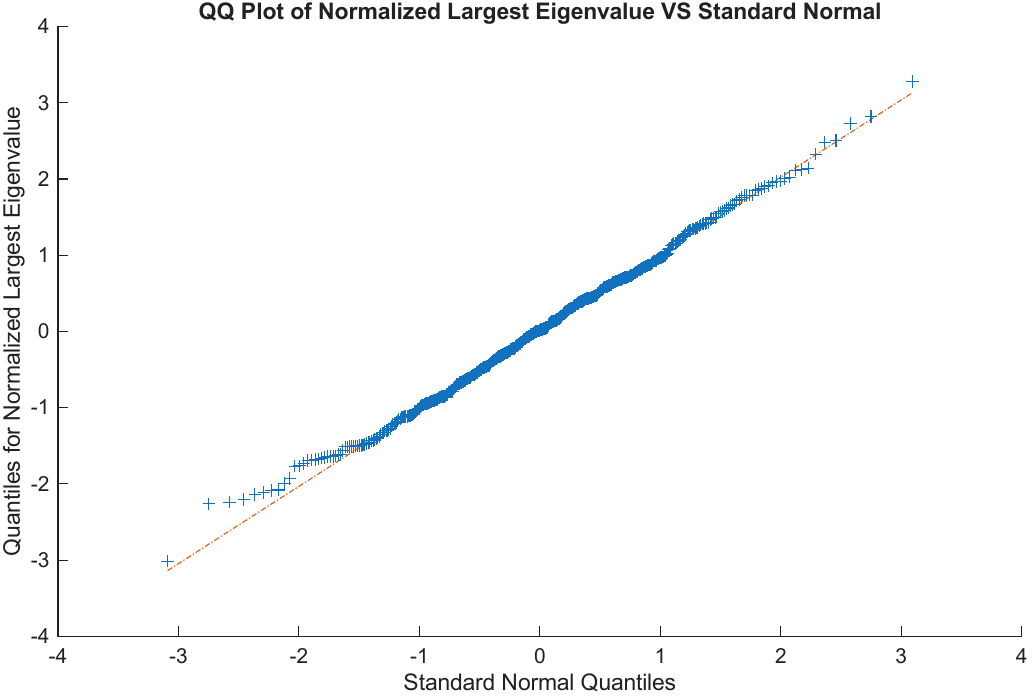}
\includegraphics[angle=0,width=0.32 \textwidth,height=0.2 \textheight]{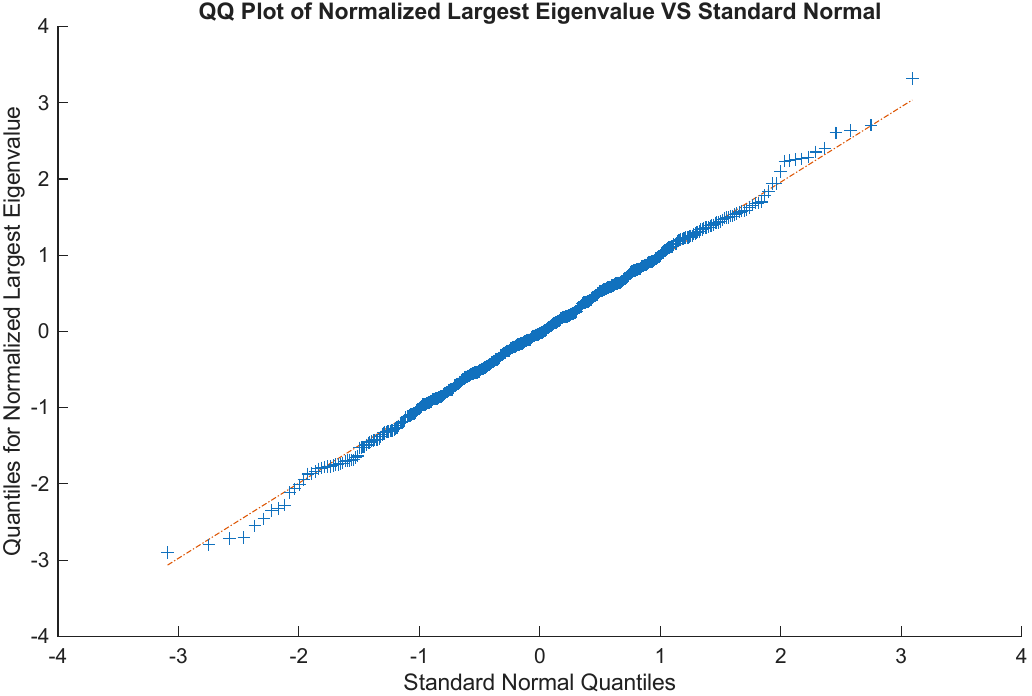}
\caption{Q--Q plots of normalized largest eigenvalue VS $\mathcal{N}(0,1)$ of 500 realizations of $\widetilde{\mathcal{G}}_{\tilde{s}}$ based on $[31,10],[63,12]$ and $[127,14]$ binary Gold codes, with $(p,n,m)=(15,31,5), (31,63,6)$ and $(63,127,7)$.}\label{fig4'}
\end{figure}
\subsection{Binary Simplex Codes $\mathscr{S}_m$ and First-Order Reed-Muller Codes $\mathrm{RM}(1,m)$}
We now present our simulation results based on $\mathscr{S}_m$ and $\mathrm{RM}(1,m)$. Recall their definitions in Section \ref{DefS}.

Our \textbf{MATLAB} simulation on these two classes of codes takes $m=14$ for the ESD, and we again generate $p=\lfloor n/2\rfloor=2^{m-1}-1$ PN sequences (so $y_n \sim 0.5$) and $2^{m-1}$ extended PN sequences (where the extended parity-check bit is inserted at the left; and $y_n=0.5$), with the first $m$ bits of each involved PN sequence chosen at random uniformly and independently respectively, using the binary primitive polynomial $f(x)=x^{14}+x^{12}+x^{11}+x+1$. As seen from Figures \ref{fig5} and \ref{fig7}, the ESD indeed looks quite different from the MP law. Moreover, we run both for 10 times and compute the KS distances of the ESD CDF with the MP CDF. The statistics are shown in Table \ref{t4}. Note in particular that the magnitude (especially the mean) is much larger than those in Table \ref{t2}.

 \begin{table}[ht!]
    \centering
    \begin{tabular}{|c||c|c|c|c|c|c|c|}
    \hline
    Code & $n$ & $p$ & mean of KS distance & variance of KS distance \\
    \hline
    \hline
    $\mathscr{S}_{14}$ & 16383 & 8191 & 0.3640693 & $1.2889 \times 10^{-5}$ \\
    \hline
    $\mathrm{RM}(1,14)$ & 16384 & 8192 & 0.36185 & $1.06 \times 10^{-5}$\\
    \hline
    \end{tabular}
    \caption{Empirical statistics of KS distances between ESD CDF of $\widetilde{\Phi}_{\tilde{s}}$ and the MP CDF in 10 realizations of numerical simulations based on $\mathscr{S}_{14}$ and $\mathrm{RM}(1,14)$}
    \label{t4}
\end{table}

As for the normalized largest eigenvalue $\tilde{\Lambda}_1$ in (\ref{NLE}), we take $m=8$ and $p=\lfloor n/2\rfloor=2^{m-1}-1$ or $2^{m-1}$, using the binary primitive polynomial $f(x)=x^8+x^4+x^3+x^2+1$ and again we run 500 realizations of $\widetilde{\mathcal{G}}_{\tilde{s}}$ in a single experiment. By Figures \ref{fig6} and \ref{fig8}, we easily see that the shape of the PDF does not look like Gaussian. For a more precise illustration, we provide the mean, variance and KS distance data as well as Q--Q plots with the standard normal distribution as shown in Table \ref{t5}, Figures \ref{fig9} and \ref{fig10} respectively. The table demonstrates that the empirical variances of $\tilde{\Lambda}_1$ for $\mathscr{S}_8$ and $\mathrm{RM}(1,8)$ are around 0.25 and 1.25, which match with the numbers in Lemma \ref{lemUS} divided by $4y_n$, and also the KS distance is far higher than the 95\% confidence threshold $\frac{K_{0.05}}{\sqrt{500}} \sim 0.060732$, while the Q--Q plots show piecewise step curves instead of a single straight-line like shape. These agree with the result in Theorem \ref{thmLS} that the true limiting distribution is discrete rather than continuous.
\begin{table}[ht!]
    \centering
    \begin{tabular}{|c||c|c|c|c|c|}
    \hline
    Code & $n$ & $p$ & mean & variance & KS distance with $\mathcal{N}(0,1)$\\
    \hline
    \hline
    $\mathscr{S}_8$ & 255 & 127 & 0.0247 & 0.2454 & 0.300606 \\
    \hline
    $\mathrm{RM}(1,8)$ & 256 & 128 & 0.031 & 1.2615 & 0.356693 \\
    \hline
    \end{tabular}
    \caption{Empirical mean, variance and KS distance with $\mathcal{N}(0,1)$ of $\tilde{\Lambda}_1$ for 500 realizations of numerical simulations based on $\mathscr{S}_8$ and $\mathrm{RM}(1,8)$}
    \label{t5}
\end{table}

\begin{figure}[ht!]
\includegraphics[angle=0,width=0.5 \textwidth,height=0.2 \textheight]{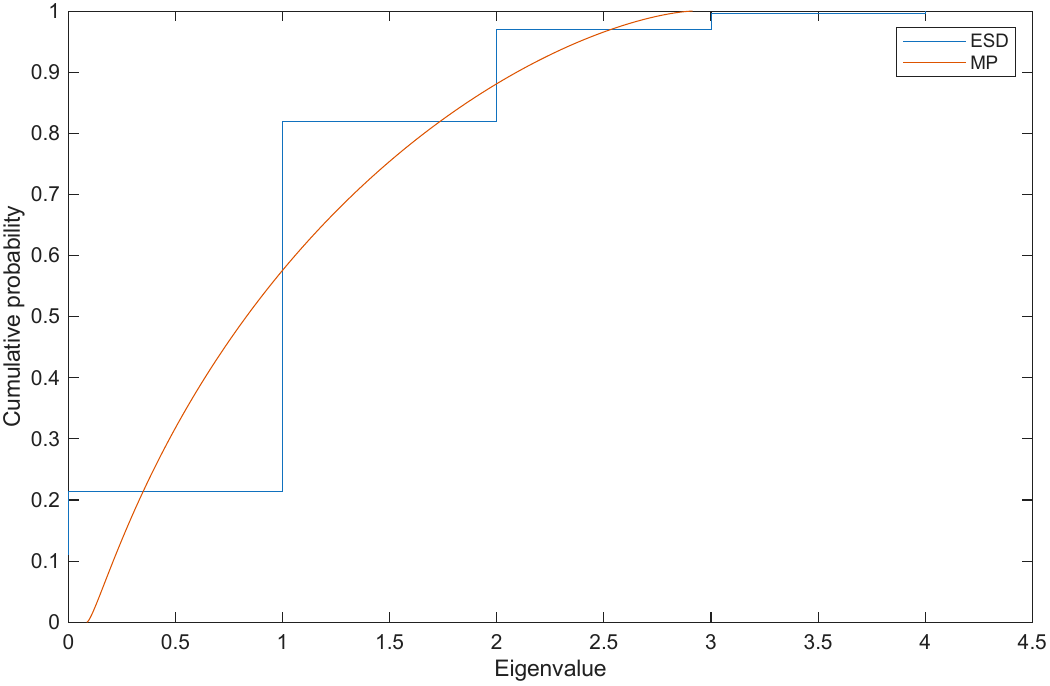}
\includegraphics[angle=0,width=0.5 \textwidth,height=0.2 \textheight]{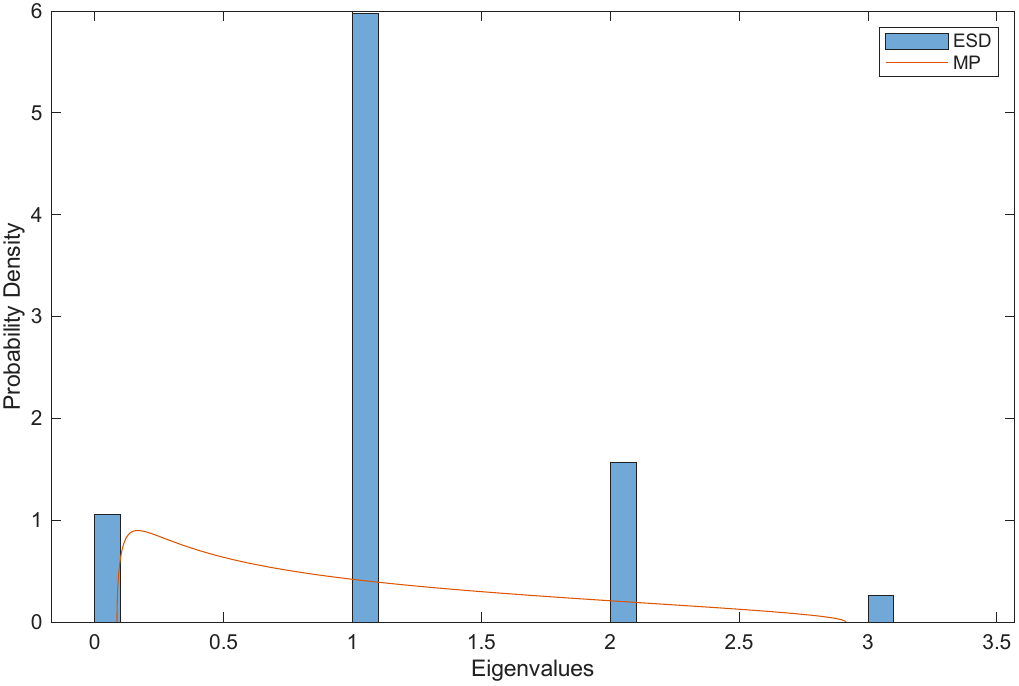}
\caption{ESD (CDF and PDF) of $\widetilde{\mathcal{G}}_{\tilde{s}}$ based on $[16383,14]$ binary simplex code, with $(p,n)=(8191,16383)$.}\label{fig5}
\end{figure}
\begin{figure}[ht!]
\begin{center}
\includegraphics[angle=0,width=0.5 \textwidth,height=0.2 \textheight]{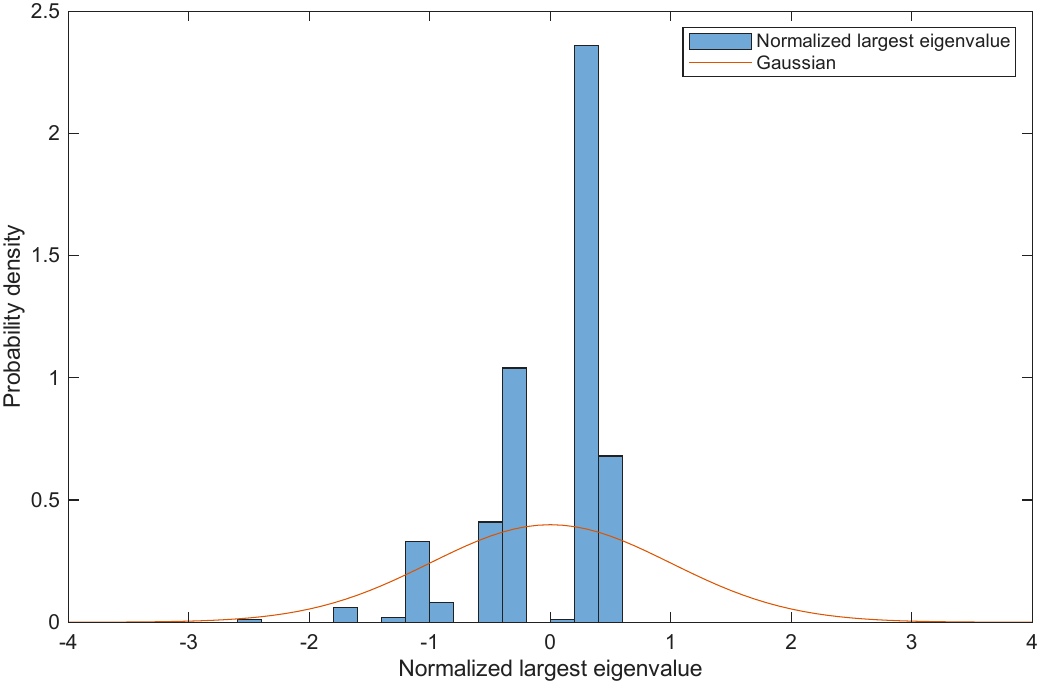}
\caption{PDF of normalized largest eigenvalue of 500 realizations of $\widetilde{\mathcal{G}}_{\tilde{s}}$ based on $[255,8]$ binary simplex code, with $(p,n)=(127,255)$.}\label{fig6}
\end{center}
\end{figure}
\begin{figure}[ht!]
\begin{center}
\includegraphics[angle=0,width=0.5 \textwidth,height=0.2 \textheight]{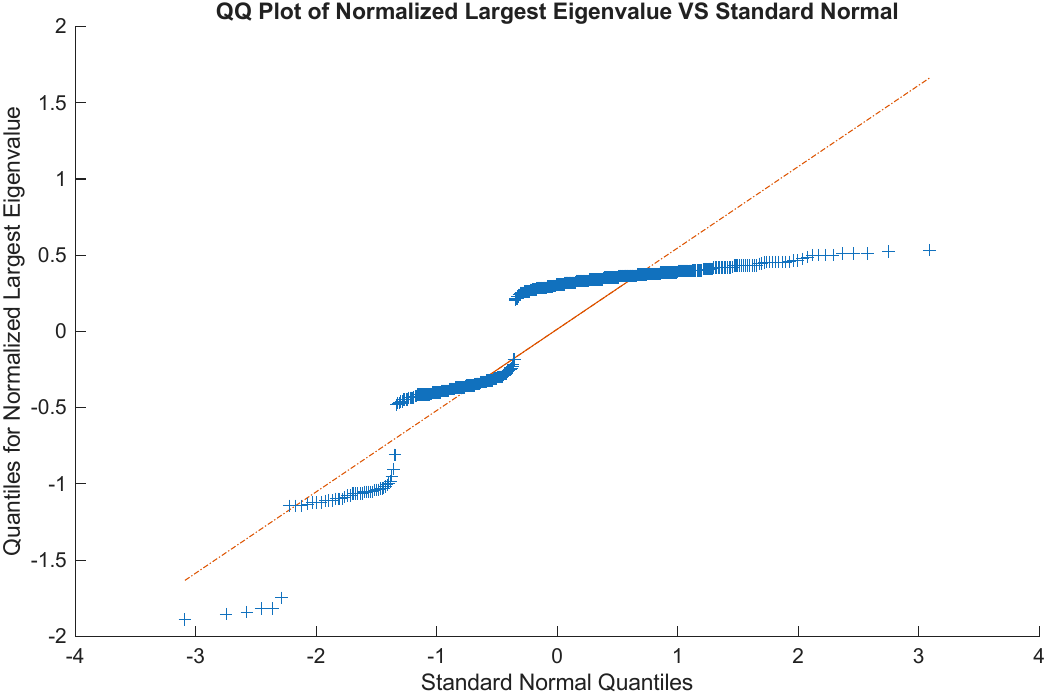}
\caption{Q--Q plot of normalized largest eigenvalue VS $\mathcal{N}(0,1)$ of 500 realizations of $\widetilde{\mathcal{G}}_{\tilde{s}}$ based on $[255,8]$ binary simplex code, with $(p,n)=(127,255)$.}\label{fig9}
\end{center}
\end{figure}

\begin{figure}[ht!]
\includegraphics[angle=0,width=0.5 \textwidth,height=0.2 \textheight]{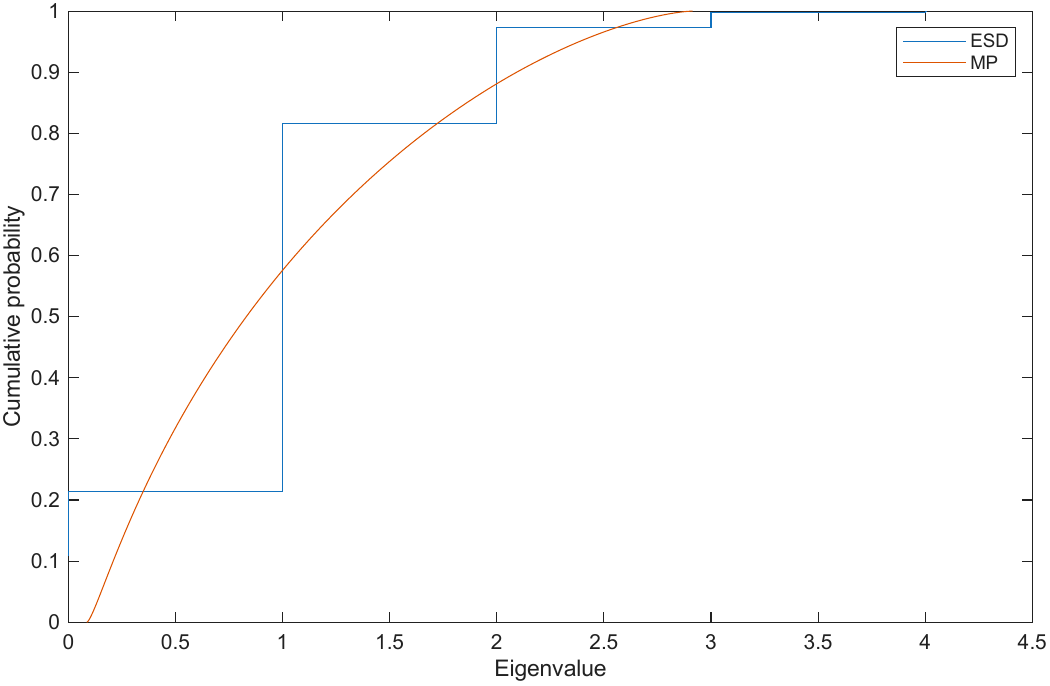}
\includegraphics[angle=0,width=0.5 \textwidth,height=0.2 \textheight]{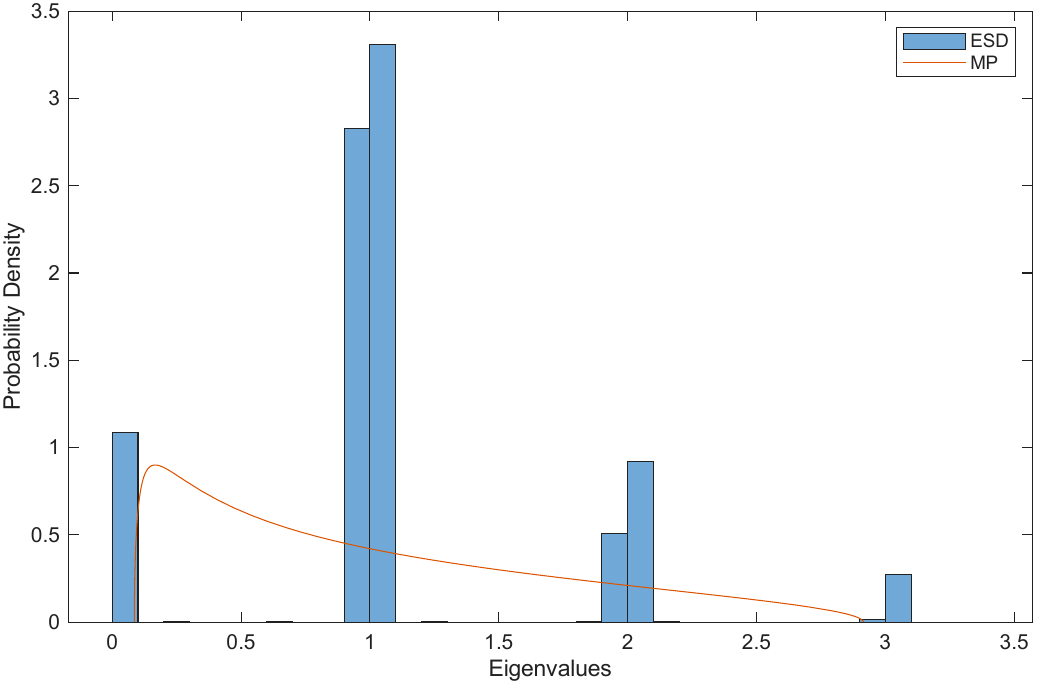}
\caption{ESD (CDF and PDF) of $\widetilde{\mathcal{G}}_{\tilde{s}}$ based on $[16384,15]$ binary first-order RM code, with $(p,n)=(8192,16384)$.}\label{fig7}
\end{figure}
\begin{figure}[ht!]
\begin{center}
\includegraphics[angle=0,width=0.5 \textwidth,height=0.2 \textheight]{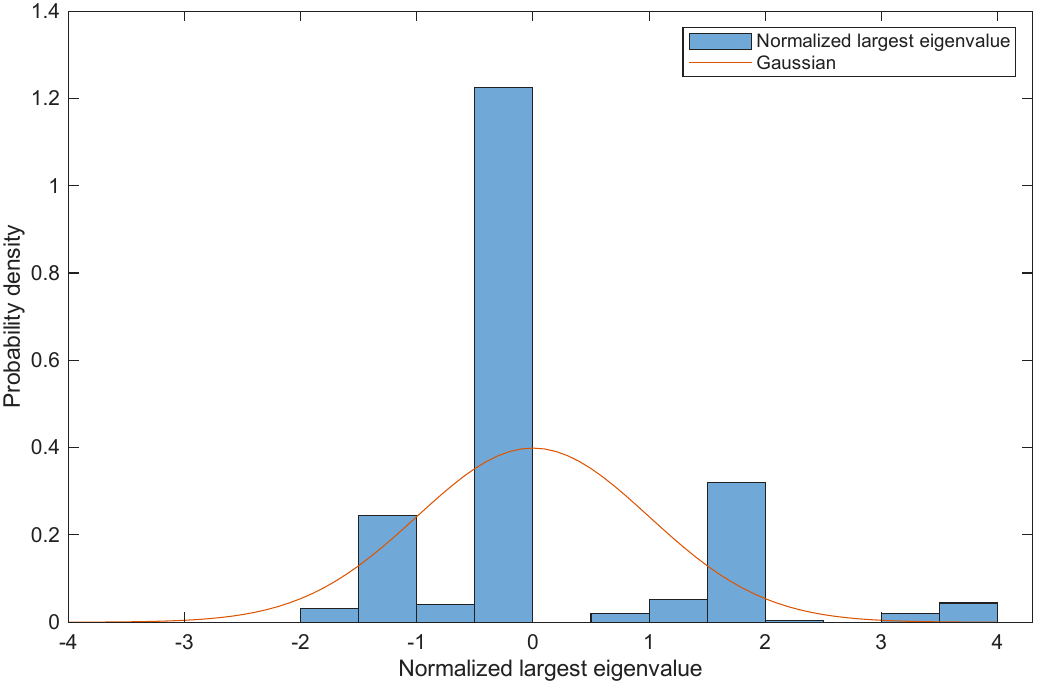}
\caption{PDF of normalized largest eigenvalue of 500 realizations of $\widetilde{\mathcal{G}}_{\tilde{s}}$ based on $[256,9]$ binary first-order RM code, with $(p,n)=(128,256)$.}\label{fig8}
\end{center}
\end{figure}
\begin{figure}[ht!]
\begin{center}
\includegraphics[angle=0,width=0.5 \textwidth,height=0.2 \textheight]{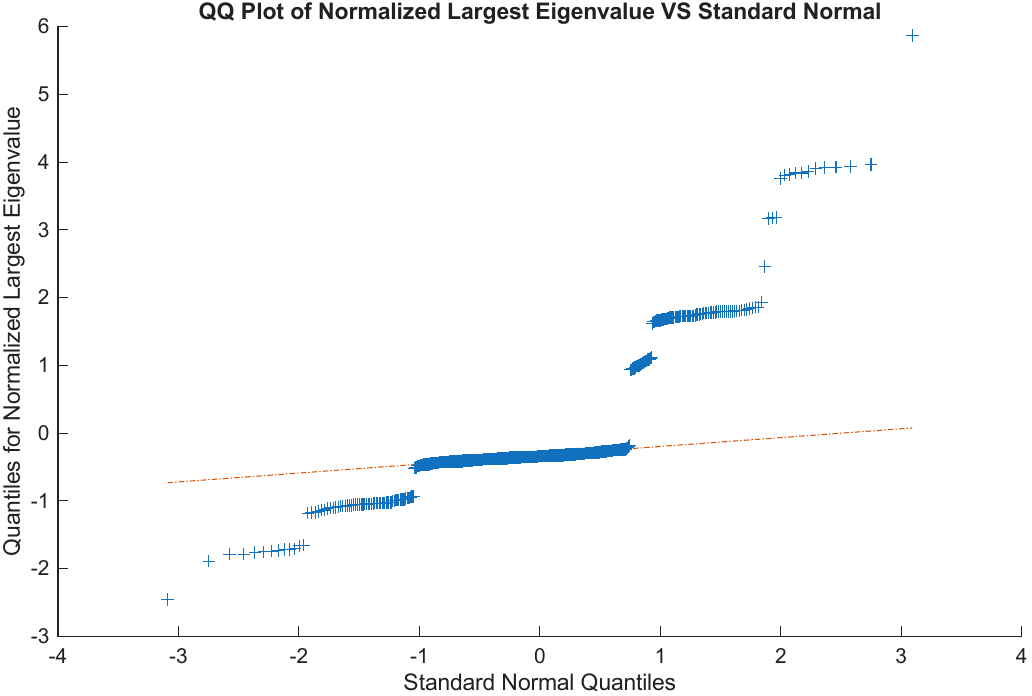}
\caption{Q--Q plot of normalized largest eigenvalue VS $\mathcal{N}(0,1)$ of 500 realizations of $\widetilde{\mathcal{G}}_{\tilde{s}}$ based on $[256,9]$ binary first-order RM code, with $(p,n)=(128,256)$.}\label{fig10}
\end{center}
\end{figure}
\section{Conclusion and Open Problems}\label{op}
In this paper we investigate sufficient conditions under which 0-1 sequences from a binary linear code $\mathcal{C}$ of length $n$ possess the following two group randomness properties through a $p \times p$ non-central random matrix $\widetilde{\mathcal{G}}$, with $p/n \to y \in (0,1)$:
\begin{enumerate}
    \item For the empirical spectral distribution (ESD), the limiting distribution is Marchenko-Pastur (MP) with high probability and convergence rate at least of order $n^{-1/4}$ provided that $d^\bot \geq 5$ (4-wise independence);
    \item For the largest eigenvalue, after centralization by $p+1+y$, the limiting distribution is Gaussian with mean 0 and variance $4y$ provided that $d^\bot \geq 3$ (pairwise independence) and $B_4(\mathcal{C})=o(n^3)$.
\end{enumerate}
Moreover, for 0-1 sequences derived from Simplex codes $\mathscr{S}_m$ and first-order Reed-Muller codes $\mathrm{RM}(1,m)$, with $d^\bot=3,4$ respectively and $B_4(\mathcal{C}) \asymp n^3$, the ESD does not converge to MP, and the largest eigenvalue has an asymptotically Poisson-type fluctuation instead.

Several interesting questions arise from the course of this paper. We discuss them below as possible directions for future research:
\begin{enumerate}
    \item We observe from our numerical experiments that the ESD of our 0-1 random matrix model based on even-degree Gold codes appears to converge to the MP law, despite the fact that they have $d^\bot=3$ and $B_4(\mathcal{C}) \asymp n^2$. By the consequence of eigenvalue interlacing in (\ref{EI2}), the ESD of the $\pm 1$ model in \cite{OQBT} should also converge to MP law, but \cite[Theorem 3]{OQBT} only yields an absolute upper bound of order 1 for the error term $E_\l$. Is it possible to improve the upper bound estimate of $E_\l$ so that it is $o_\l(1)$ for all linear codes with $d^\bot \geq 3$ and $B_4(\mathcal{C})=O(n^2)$ or even the same condition as Theorem \ref{mainthm2'}, that is, $B_4(\mathcal{C})=o(n^3)$, so as to cover the phenomenon on the even-degree Gold codes as well?
    \item While we have derived a convergence rate of the ESD to the MP law as $O_\prec(n^{-1/4})$ under the mere condition $d^\bot \geq 5$ in Theorem \ref{mainthm}, our numerical simulation on the Gold codes suggests a tighter bound, around of order $n^{-3/4}$. What is the true optimal order of convergence rate under the following conditions?
    \begin{enumerate}
        \item $d^\bot \geq 5$;
        \item $d^\bot \geq 3$ and $B_4(\mathcal{C})=o(n^2)$, such that convergence of ESD to MP is guaranteed by \cite[Theorem 3]{OQBT} together with the consequence of eigenvalue interlacing in (\ref{EI2});
        \item $d^\bot \geq 3$ and $B_4(\mathcal{C})=O(n^2)$ or even $o(n^3)$ as conjectured in 1., assuming it is true?
    \end{enumerate}
    As by Remark \ref{Rk1}, this rate cannot be better than $O_\prec(n^{-1})$.
        \item The condition $d^\bot \geq 5$ appears as a sufficient condition for many results regarding ESD of random matrices from linear codes: the $\pm 1$ Gram matrix model $\mathcal{G}$ with $p/n \to y \in (0,1)$ \cite{CTX,OQBT}, the $\pm 1$ re-normalized Gram matrix model with $p \to \infty$ and $p/n \to 0$ \cite{CKX}, and the 0-1 Gram matrix model $\widetilde{\mathcal{G}}$ studied in this paper. Can this be extended to code-based random matrix models with other symmetries (see \cite{Bose,Bose2})? As a comparison, for linear spectral statistics of monomials of the $\pm 1$ Gram matrix model $\mathcal{G}$, the best known sufficient condition such that the central limit theorem holds, however, is $d^\bot \geq 7$ instead (see \cite{CLT}).
\end{enumerate}
\section*{Acknowledgments}
The author is grateful to the anonymous referees for their careful reading of the manuscript and for their many invaluable suggestions, which have improved the quality, strength, and mathematical insight of this work.

\section*{Funding Sources}
This research was supported by the author's research start-up fund provided by Hetao Institute of Mathematics and Interdisciplinary Sciences.

\section*{Declaration of Generative AI and AI-Assisted Technologies in the Manuscript Preparation Process}
During the preparation of this work, the author used DeepSeek-V4 and Gemini to seek advice for better knowledge how to motivate the work in the introduction. In the revision, the author further used Claude Sonnet 5.0 to seek advice for initiating the idea in the proofs of Lemmas \ref{lemU} and \ref{lemLUS} in Section \ref{plemLUS}, and used DeepSeek-V4 for suggestions to \textbf{MATLAB} codes that generate KS distance data. DeepSeek-V4, Claude Sonnet 5.0/5.5 and ChatGPT-5.6 Luna were also used for assistance in further proofreading and reorganization to address referees' suggestions. The author takes full responsibility for the content of the manuscript.

\bibliographystyle{IEEEtranS}
\bibliography{CCH}

\end{document}